\documentclass[11pt]{amsart}
\pdfoutput=1

\usepackage{amsmath,amssymb,amsthm}
\usepackage{graphicx}

\newtheorem{theorem}{Theorem}[section]
\newtheorem{proposition}[theorem]{Proposition}
\newtheorem{corollary}[theorem]{Corollary}
\newtheorem{lemma}[theorem]{Lemma}

\newtheorem{introthm}{Theorem}

\newtheorem*{the_example}{Theorem \ref{introthm:Example}}

\theoremstyle{definition}

\newtheorem*{question}{Question}

\theoremstyle{remark}
\newtheorem{remark}[theorem]{Remark}

\newcommand{\thmref}[1]{Theorem~\ref{#1}}
\newcommand{\propref}[1]{Proposition~\ref{#1}}
\newcommand{\secref}[1]{\S\ref{#1}}
\newcommand{\secsref}[2]{\S\S\ref{#1}--\ref{#2}}
\newcommand{\lemref}[1]{Lemma~\ref{#1}}
\newcommand{\corref}[1]{Corollary~\ref{#1}}
\newcommand{\figref}[1]{Fig.~\ref{#1}}

\DeclareMathOperator{\area}{area}
\DeclareMathOperator{\diam}{diam}
\DeclareMathOperator{\gr}{gr}
\DeclareMathOperator{\Gr}{Gr}
\DeclareMathOperator{\std}{std}

\DeclareMathOperator{\up}{\pi_U}
\DeclareMathOperator{\down}{\pi_D}
\DeclareMathOperator{\Hom}{Hom}
\DeclareMathOperator{\I}{i}

\newcommand{\emul}{\stackrel{{}_\ast}{\asymp}}
\newcommand{\gmul}{\stackrel{{}_\ast}{\succ}}
\newcommand{\lmul}{\stackrel{{}_\ast}{\prec}}
\newcommand{\eadd}{\stackrel{{}_+}{\asymp}}
\newcommand{\gadd}{\stackrel{{}_+}{\succ}}
\newcommand{\ladd}{\stackrel{{}_+}{\prec}}
\newcommand{\graph}{\theta}

\newcommand{\identity}{\mathrm{id}}

\newcommand{\calA}{{\mathcal A}}

\newcommand{\calE}{{\mathcal E}}
\newcommand{\calI}{{\mathcal I}}

\newcommand{\calR}{{\mathcal R}}

\newcommand{\calT}{{\mathcal T}}
\newcommand{\T}{{\mathcal T}}
\newcommand{\F}{{\mathcal F}}

\newcommand{\ML}{\mathcal{ML}}
\newcommand{\PML}{\mathcal{PML}}
\newcommand{\G}{{\mathcal G}}

\newcommand{\ts}{{{\mathcal T}(S)}}
\newcommand{\QF}{{\mathcal{QF}}}
\newcommand{\PSL}{\mathrm{PSL}}
\newcommand{\hol}{\mathrm{hol}}
\newcommand{\X}{\mathcal{X}}
\newcommand{\CP}{\mathbb{CP}}
\renewcommand{\P}{\mathcal{P}}
\newcommand{\sslash}{/\!\!/}
\newcommand{\pl}{\mathrm{Pl}}
\newcommand{\base}{O}
\renewcommand{\vec}[1]{\overrightarrow{{#1}}}
\newcommand{\qd}{\mathcal{Q}}

\newcommand{\dhyp}{d_{\mathrm{hyp}}}
\newcommand{\rhohyp}{\rho_{\mathrm{hyp}}}
\newcommand{\rhoth}{\rho_{\mathrm{Th}}}
\newcommand{\hyplap}{\Delta_{\mathrm{hyp}}}

\newcommand{\im}{\mathrm{Im}}
\newcommand{\re}{\mathrm{Re}}

\renewcommand{\bar}{\overline}
\renewcommand{\tilde}{\widetilde}
\renewcommand{\Tilde}{\widetilde}

\newcommand{\param}{{\mathchoice{\mkern1mu\mbox{\raise2.2pt\hbox{$
\centerdot$}}
\mkern1mu}{\mkern1mu\mbox{\raise2.2pt\hbox{$\centerdot$}}\mkern1mu}{
\mkern1.5mu\centerdot\mkern1.5mu}{\mkern1.5mu\centerdot\mkern1.5mu}}}

\newcommand{\Teich}{Teich\-m\"uller }
\newcommand{\Extension}{{\tt \delta}}

\newcommand{\FinalBound}{{\tt K}}
\newcommand{\Lipschitz}{{\tt L}}
\newcommand{\distance}{{\tt d}}
\newcommand{\Distance}{{\tt D}}

\newcommand{\height}{{\tt h}}
\newcommand{\Haus}{{\tt C}}
\newcommand{\measure}{{\tt m}}

\newcommand{\length}{{\tt l}}

\newcommand{\LipBound}{{\tt B}}

\newcommand{\quasi}{{\tt q}}

\newcommand{\Height}{{\tt H}}
\newcommand{\Length}{{\tt L}}
\newcommand{\Measure}{{\tt M}}

\newcommand{\vertical}{{\mathcal V}}
\newcommand{\horizontal}{{\mathcal H}}

\newcommand{\Xnew}{\hat{X}}

\newcommand\C{{\mathbb C}}

\newcommand\N{{\mathbb N}}
\newcommand\R{{\mathbb R}}

\newcommand\from{{\colon \,}}

\newcommand{\hyp}{{\mathbb H}}

\newcommand{\Xstd}{X_{\std}}
\newcommand{\Ustd}{U_{\std}}

\newcommand{\so}{\omega_{\std}}

\newcommand{\bvp}{{\overline \varphi}}

\newcommand{\bso}{{\overline \omega_{\std}}}
\newcommand{\bU}{{\overline U}}
\newcommand{\bH}{{\overline \horizontal}}
\newcommand{\btheta}{{\overline \theta}}
\newcommand{\bnu}{{\overline \nu}}

\newcommand{\ep}{\epsilon}

\begin{document}

\title{Grafting Rays Fellow Travel Teichm\"uller Geodesics}
\author{Young-Eun Choi}
\email{youngeun.choi@stanford.edu}
\author{David Dumas}
\thanks{The second author was partially supported by the NSF through
 DMS-0402964 and DMS-0805525.}
\email{ddumas@math.uic.edu}
\urladdr{http://www.math.uic.edu/\textasciitilde ddumas/}
\author{Kasra Rafi}
\email{rafi@math.ou.edu}
\urladdr{http://www.math.ou.edu/\textasciitilde rafi/}
\date{April 17, 2011}

\begin{abstract}
Given a measured geodesic lamination on a hyperbolic
  surface, grafting the surface along multiples of the lamination
  defines a path in Teichm\"uller space, called the grafting ray. We
  show that every grafting ray, after reparameterization, is a
  Teichm\"uller quasi-geodesic and stays in a bounded neighborhood of 
  some Teichm\"uller geodesic.

  As part of our approach, we show that grafting rays have controlled
  dependence on the starting point. That is, for any measured geodesic
  lamination $\lambda$, the map of Teichm\"uller space which is
  defined by grafting along $\lambda$ is $\Lipschitz$--Lipschitz with
  respect to the Teichm\"uller metric, where $\Lipschitz$ is a
  universal constant. This Lipschitz property follows from an
  extension of grafting to an open neighborhood of Teichm\"uller space
  in the space of quasi-Fuchsian groups.
\end{abstract}

\maketitle

\section{Introduction}
Let $S$ be a closed surface with finite genus, possibly with finitely many punctures.
Let $X$ be a point in \Teich space $\ts$, and let $\lambda$ be a measured geodesic 
lamination on $X$ of compact support.  The pair $X$ and the projective class 
$[\lambda]$ determines a \Teich geodesic ray which starts at $X$ and where 
the associated vertical foliation is
a multiple of $\lambda$ \cite{hubbard-masur}, \cite{kerckhoffasymp}.  Let 
$\G(t,\lambda, X)$, $t \geq 0$, denote the point on this ray whose \Teich 
distance from $X$ is $t$. The pair $X$ and $\lambda$ determines another 
ray in $\ts$ defined by grafting $X$ along $s \lambda$, for $s \in \R_+$.  We 
denote the resulting Riemann surface by $\gr(s\lambda,X)$.  
(See \cite{kamishima-tan}, \cite{tanigawa}, \cite{mcmullen} for background
on grafting.)

In this paper we show that each grafting ray stays in a bounded
neighborhood of a \Teich geodesic:

\begin{introthm}
\label{thm:Main}
Let $X \in \T(S)$ be $\epsilon$--thick and let $\lambda$ be a measured
geodesic lamination on $X$ with unit hyperbolic length.  Then for all
$t \geq 0$ we have
$$ 
d_\calT \Big ( \gr \big(e^{2t}\lambda, X \big), \G \big(t,\lambda,X
\big) \Big) 
   \leq \FinalBound
$$
where the constant $\FinalBound$ depends only on $\ep$ and the 
topology of $S$ (it is independent of $\lambda$ and $X$). 
\end{introthm}
Here $d_\calT$ is the \Teich distance 
and we say that $X \in
\T(S)$ is \emph{$\epsilon$--thick} if the injectivity radius of the hyperbolic
metric on $X$ is at least $\epsilon$ at every point.

The proof of Theorem \ref{thm:Main} actually produces a bound on the distance 
that depends continuously on the point in moduli space determined by $X$;
the existence of a constant depending only on the injectivity radius
is then a consequence of the compactness of the $\epsilon$--thick part
of moduli space.
However, the dependence on the injectivity radius of $X$ is
unavoidable:

\begin{introthm} \label{introthm:Example} There exists a sequence of
  points $X_n$ in $\ts$ and measured laminations $\lambda_n$ with unit
  hyperbolic length on $X_n$ such that for any sequence $Y_n$ in
  $\ts$,
$$\sup_{n,{t \geq 0}} d_\calT \left(\gr(e^{2t} \lambda_n, X_n), \G(t,\lambda_n,Y_n) \right ) 
= \infty.$$
\end{introthm}

We now outline the proof of \thmref{thm:Main}.  The main construction
(carried out in \secsref{sec:Dual}{sec:TheMap}, culminating in \propref{prop:TheMap}) produces an explicit family of quasiconformal maps
between the Riemann surfaces along a grafting ray and those of a
\Teich geodesic ray starting from another point $Y \in \T(S)$.  This
is done in the case that $S$ has no punctures.  Unfortunately, the
quasiconformal constant for these maps and the distance $d_\T(X,Y)$
depend on the pair $(X,\lambda)$, whereas for the main theorem we
seek a uniform upper bound.

Such a bound is derived from the construction in several steps. First,
we show that there exist points $\Xstd \in \T(S)$ for which the
quasiconformal constants are uniform over an open set in $\ML(S)$
(\secref{sec:TheMap}).  Then, using the action of the mapping class
group and the co-compactness of the $\epsilon$--thick part of
$\T(S)$, we show that for any $\epsilon$--thick surface $X$ and any
$\lambda \in \ML(S)$ there exists $\Xstd$ near $X$ 
for which the uniform estimates
apply to $(\Xstd, \lambda)$.

At this point we have proved the main theorem up to 
moving the base points of both the grafting ray and the \Teich
geodesic ray by a bounded distance from the given $X \in \T(S)$. 
The proof is concluded by showing that both the grafting and \Teich
rays starting from these perturbed
basepoints fellow travel those starting from the original point $X$.

For the Teichm\"uller ray case, we use a recent theorem of Rafi
\cite{rafi3} (generalizing earlier results of \cite{masur} and
\cite{ivanov}) which states that \Teich geodesics with the same
vertical foliation fellow travel, with a bound on the distance depending
only on the thickness $\epsilon$ and the distance between the starting points.

It remains to show that grafting rays in a given direction $\lambda$
fellow travel.  In \secref{sec:Lipschitz} we show that
$\lambda$--grafting defines a self-map of \Teich space that is
uniformly Lipschitz with respect to the \Teich metric, and since the
Lipschitz constant is independent of $\lambda$, the fellow traveling
property of grafting rays follows.  The key to this Lipschitz bound is
a certain extension of grafting to quasi-Fuchsian groups.

In order describe this extension, we regard grafting as a map $$\gr:
\ML(S) \times \F(S) \to \T(S),$$where $\ML(S)$ is the space of
measured laminations on $S$ and $\F(S) \simeq \T(S)$ is the
realization of \Teich space as the set of marked Fuchsian groups,
which is a real-analytic manifold parameterizing hyperbolic structures
on $S$.  By a construction of Thurston, this map lifts to a
projective grafting map $\Gr \from \ML(S) \times \F(S) \to \P(S)$, where
$\P(S)$ is the space of marked complex projective structures.

Interpreting \Teich space as the ``diagonal'' in quasi-Fuchsian space
$\QF(S) \simeq \T(S) \times \bar{\T(S)}$, we show that projective
grafting extends to a holomorphic map defined on a uniform metric
neighborhood of $\F(S)$ in $\QF(S)$.  Here we give $\QF(S)$ the
Kobayashi metric, which is the sup-product of the Teichm\"uller
metrics on $\T(S)$ and $\bar{\T(S)}$, and we show:

\begin{introthm}
\label{thm:extension}
There exists $\Extension > 0$ such that projective grafting extends to
a map $\Gr \from \ML(S) \times \QF_\Extension(S) \to \P(S)$ that is
holomorphic with respect to the second parameter, where
$\QF_\delta(S)$ is the open $\delta$--neighborhood of $\F(S)$ with
respect to the Kobayashi metric on $\QF(S)$.
\end{introthm}

Composing this map $\Gr$ with the forgetful map $\pi \from \P(S) \to
\T(S)$, we also obtain a map $\gr = \pi \circ \Gr \from
\QF_\Extension(S) \to \T(S)$, holomorphic in the first factor, which
is the extension that we use in the proof of \thmref{thm:Main}.
Note that the original grafting map $\gr \from \ML(S) \times \T(S) \to \T(S)$ is
\emph{not} holomorphic with respect to the usual complex
structure on $\T(S)$; the holomorphic behavior described in
\thmref{thm:extension} can only be seen by considering \Teich space as a totally real submanifold of $\QF(S)$.

We remark that the existence of a local holomorphic extension of
$\Gr(\lambda, \param)$ (or $\gr(\lambda,\param)$) to a neighborhood of a
point in $\F(S)$ follows easily from results of Kourouniotis or
Scannell--Wolf (see \secref{subsec:extension} for details), but that
extension to a uniform neighborhood of $\F(S)$ (i.e.~the existence of
$\Extension$) is essential for application to \thmref{thm:Main} and
does not follow immediately from such local considerations.

Using the holomorphic extension of grafting and the contraction of
Kobayashi distance by holomorphic maps, we then establish the
Lipschitz property for grafting:

\begin{introthm}
\label{thm:Lipschitz}
There exists a constant $\Lipschitz$ such that for any measured lamination
$\lambda \in \ML(S)$, the grafting map $\gr_\lambda \from \ts \to \ts$ is
$\Lipschitz$--Lipschitz. That is, given any two 
points $X$ and $Y$ in $\ts$, we have
$$
d_{\calT} \big(\gr(\lambda,X),\gr(\lambda, Y) \big) 
  \leq \Lipschitz \, d_{\calT}(X,Y).
$$
\end{introthm}

In \secref{sec:Conclusion} we combine the rectangle construction with
the fellow traveling properties for grafting and \Teich rays to derive
the main theorem for compact surfaces.  In \secref{sec:Punctures} we
show how the preceding argument can be modified to prove
\thmref{thm:Main} in the case $S$ has punctures.

Finally, in \secref{sec:Example} we construct an example illustrating
\thmref{introthm:Example}.

\subsection*{Shadows in the curve complex}
Given any point $X \in \ts$ and a projective
class $[\lambda]$ of a measured geodesic lamination on $X$, there are
different ways to geometrically define a ray which starts at $X$ and
``heads in the direction of $[\lambda]$''; examples include the \Teich
ray, the grafting ray, and the line of minima \cite{kerckhofflom}.
Given any path in \Teich space, by taking the shortest curve on each
surface, we get a path in the complex $\mathcal C(S)$ of curves of
$S$, which is often called the {\em shadow} of the original path.
Masur and Minsky showed \cite{masur-minsky} that the shadow of a
\Teich geodesic is an unparameterized quasi-geodesic in $\mathcal
C(S)$.  A consequence of \thmref{thm:Main} is that the shadow of a
grafting ray remains a bounded distance in $\mathcal C(S)$ from the
shadow of a \Teich geodesic ray. Hence, it follows that the same is
true of the grafting ray.  In the case of a line of minima, though it
may not remain a bounded distance from any \Teich geodesic, it was
shown \cite{crs1} that its shadow fellow-travels that of its
associated \Teich geodesic.  It is interesting that although these
paths are defined in rather different ways, at the level of the curve
complex, they are essentially the same.

\subsection*{Related results and references}
Using a combinatorial model for the Teichm\"uller metric, D\'{i}az and
Kim showed that the conclusion of \thmref{thm:Main} holds for grafting
rays of laminations supported on simple closed geodesics
\cite{diaz-kim}.  However, the resulting bound on distance depends on
the geometry of the geodesics in an essential way, obstructing the
extension of their method to more general laminations by a limiting
argument.

Grafting rays were also studied by Wolf and the second author in
\cite{dumas-wolf}, where it was shown that for any $X \in \T(S)$, the
map $\lambda \mapsto \gr(\lambda,X)$ gives a homeomorphism between
$\ML(S)$ and $\T(S)$.  In particular, \Teich space is the union
of the grafting rays based at $X$, which are pairwise disjoint.
 In light
of \thmref{thm:Main}, we find that this ``polar coordinate system''
defined using grafting is a bounded distance from the Teichm\"uller
exponential map at $X$.

\subsection*{Acknowledgments}

Some of this work was completed at the Mathematical Sciences Research
Institute during the Fall 2007 program ``Teichm\"uller Theory and
Kleinian Groups''.  The authors thank the institute and the organizers
of the program for their hospitality.  They also thank the referees
for helpful comments which improved the paper.

\section{The Orthogonal Foliation to a Lamination}
\label{sec:Dual}

Throughout \secsref{sec:Dual}{sec:Conclusion} we assume that $S$ has
no punctures. In this section we construct, for every $X \in \T(S)$
and every measured lamination $\lambda$, a measured foliation
$\horizontal(\lambda,X)$ orthogonal to $\lambda$ in $X$.  This is a
kind of approximation for the horizontal foliation of $\G(t,\lambda,
X)$, which we do not explicitly know.  In the case where $\lambda$ is
maximal (that is, the complement of $\lambda$ is a union of ideal
triangles) the measured foliation is equivalent to the horocyclic
foliation constructed by Thurston in \cite{thurston:minimal-stretch}.

\subsection*{A measured foliation orthogonal to $\lambda$} \label{sec:nu} 

Let $g$ be a geodesic in $\hyp^2$. Consider the closest point  projection map onto $g$,
which takes each point in $\hyp^2$ to the point on $g$ to which it is closest.
The fibers of the projection foliate $\hyp^2$ by geodesics
perpendicular  to $g$. Analogously,  if $\{g_i\}$  is a collection of
disjoint  geodesics, then the closest point projection to $\cup g_i$ is well-defined,
except  at the points that are equidistant
to two  or more geodesics in $\{g_i\}$. These points form a (possibly disconnected) graph  
where the edges are geodesic segments, rays, or lines. 
The fibers of the projection foliate $\hyp^2$ by piecewise
geodesics.  

The lamination $\lambda$ lifts to a set 
$\tilde \lambda$ of disjoint infinite geodesics which is invariant under deck 
transformations. 
 Let $\tilde \graph$ be 
the graph of points where the closest point projection to $\tilde \lambda$ 
is not well-defined and let $\graph$ be the projection of $\tilde
\graph$ to $X$. We call
the set $\graph=\graph(\lambda,X)$ the {\em singular locus} of the closest point 
projection
map.  As above, the fibers defined by the projection provides 
a foliation of $\hyp^2$ which projects down to a
foliation of 
$X$. Thus we obtain a singular foliation 
$\horizontal=\horizontal(\lambda, X)$ on $X$ orthogonal to $\lambda$.  The foliation has
singularities at 
the vertices of $\graph$, where the number of prongs at a singularity 
coincides with the valence of the vertex. The leaves of $\horizontal$ are piecewise 
geodesics whose non-smooth points lie on $\graph$.  For later
purposes, we prefer to maintain the non-smooth structure of $\horizontal$
along $\graph$. A leaf of $\horizontal$ that joins two vertices is
called a {\em saddle connection}.

\begin{proposition} \label{prop:hyp-dual}
The hyperbolic arc-length along $\lambda$ induces a transverse measure on 
$\horizontal(\lambda,X)$.
\end{proposition}
\begin{proof}
Let $\eta$ be a smooth embedded arc in $X$. 
 First suppose that 
the interior of $\eta$ is contained in a component of $X \setminus \graph$ 
(the endpoints of $\eta$ may be contained in $\graph$) and that at every
  interior point, $\eta$ is transverse to $\horizontal$.  Let $\Tilde{\eta}$ be a
  lift of $\eta$. Then the closest point projection onto $\tilde
  \lambda$ projects $\Tilde{\eta}$ to an arc on a leaf of $\tilde \lambda$.
  (In the case that an endpoint of $\Tilde{\eta}$ is in $\tilde \graph$, project
  the endpoint to the same leaf of $\tilde \lambda$ as the interior
  points.)  Define the measure on $\eta$ to be the length of this arc.
  In the case that $\eta$ is contained in $\graph$, observe that although
  the closest point projection of $\Tilde{\eta}$ to $\tilde \lambda$ is not
  well-defined, any choice of projection has the same length because
  $\tilde \graph$ is equidistant to the corresponding leaves of $\tilde
  \lambda$.  Define the measure on $\eta$ to be the length of this
  arc.  If the interior of $\eta$ intersects $\graph$ at a point $p$, we
  say $\eta$ is {\em transverse} to $\horizontal$ at $p$ if, on a small circle
  $C$ centered at $p$, the points of $\eta \cap C$ separate the points
  of $\ell \cap C$, where $\ell$ is the leaf of $\horizontal$ through $p$. In
  general, if $\eta$ is transverse to $\horizontal$ at every point, we define
  the $\horizontal$-measure of $\eta$ to be the sum of the 
  $\horizontal$--measures of the subarcs $\cup [\eta \cap (X \setminus \theta)]$ 
  and any subarcs in $\eta \cap
  \theta$.  In this way, we equip $\horizontal$ with a transverse measure that
  coincides with arc-length along $\lambda$.
  \end{proof} 
   
   Note that neither $\horizontal$ nor its transverse measure depends in any way on the
measure on $\lambda$.

\section{Rectangle decomposition}
\label{sec:decomposition}
We now describe a decomposition of $X$ into \emph{rectangles} using 
the orthogonal foliation $\horizontal=\horizontal(\lambda,X)$. 

Choose an arc $\omega$ contained in $\theta$ that contains no
singularities of $\theta$.
For every point $p$ on $\omega$ and a choice of normal direction to $\omega$, 
consider the arc of $\horizontal$ starting from $p$ in that direction.
By  Poincar\'e recurrence (see for example \cite[\S 5.1]{FLP}), this arc either
ends at a singular point of $\theta$ or  intersects
$\omega$ again.  We call such an arc \emph{exceptional} if it intersects a singular 
point of $\theta$ or an endpoint of $\omega$ before intersecting the interior of 
$\omega$.  In particular we consider any arc starting from an endpoint of $\omega$ 
to be exceptional.  Let $P \subset \omega$ denote the set of endpoints of 
exceptional arcs.  Label the normal directions of $\omega$ as $n_1,n_2$.  Then 
$P = P_1 \cup P_2$, where $P_i$ corresponds to endpoints of exceptional arcs 
in the normal direction $n_i$. 

The first return map of $\horizontal$ is defined on the set of pairs $(p,n_i)$ where 
$p \in \omega$, $i \in \{1,2\}$, and $p \notin P_i$.  That is, the first return map is 
naturally a self-map of
$$
\big(\omega \setminus P_1\big) \times \{n_1 \} \cup
\big(\omega \setminus P_2 \big) \times \{n_2 \}.
$$
An open interval of $\omega \setminus P_i$ flows along $\horizontal$ in the direction 
$n_i$ until it returns to another (possibly overlapping) interval of $\omega$ 
sweeping out a \emph{rectangle} that has two edges attached to $\omega$ 
(see Figure~\ref{fig:rectangle}).  
We refer to it as a rectangle, despite the fact that the edges along $\horizontal$
are jagged, because the endpoints of the edges which are attached to
$\omega$ give four distinguished points on the boundary. We call these
points the \emph{vertices} of the rectangle.
\begin{figure}[tb]
\begin{center}
\includegraphics{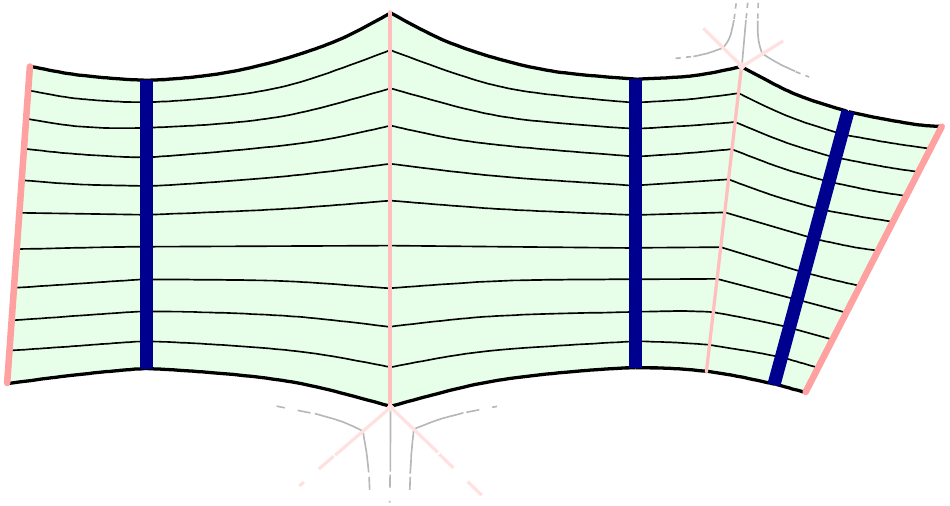}
\caption{A rectangle $R$ foliated by leaves of $\horizontal$ in the case
  where $\lambda$ is a weighted multicurve (of which three segments
  are shown in bold).  The leaves of $\horizontal$ are piecewise geodesics, smooth
  away from $\graph$ and orthogonal to $\lambda$ at each
  intersection. The left and right edges of $R$ are contained in $\omega$.
  \label{fig:rectangle}}
\end{center}
\end{figure}

If $\horizontal(\lambda, X)$ has no saddle connections, 
this decomposes $X$ into a union of rectangles.  If $\horizontal(\lambda, X)$ has saddle connections, then
the union of the rectangles may only be a subsurface of $X$, whose
boundary is made of saddle connections.  
We will however, assume below in (H1) that $\horizontal(\lambda, X)$ has no saddle connections.

The interiors of the rectangles are disjoint and contain no singularities.
For every rectangle $R$, we call the pair of opposite edges that 
are sub-arcs of $\omega$ the \emph{vertical edges} of $R$.
The refer to the other pair of opposite edges,
that are sub-arcs of leaves of $\horizontal$, as the \emph{horizontal edges} of $R$. 
The rectangle decomposition obtained from a transversal
$\omega$ in this way will be denoted by $\mathcal R(\omega,\lambda,X)$.

\subsection{Topological stability of the decomposition}
\label{sec:stability}
 
Suppose $\mu$ is a maximal measured lamination, that is, the complement of 
$\mu$ is a union of ideal triangles. The foliation $\mathcal H(\mu,X)$ has a 
three-prong singularity at the center of each ideal triangle.  If $\lambda$ is close 
to $\mu$ in the usual weak topology of $\ML(S)$ then, because $\mu$ is maximal, 
the lamination $\lambda$ is also close to $\mu$ in the Hausdorff
topology (see \cite[pp.~24--25]{thurston:minimal-stretch}
\cite{weiss}).  Since $\mathcal H (\lambda,X)$ varies
continuously with the support of $\lambda$, the singularities of
$\mathcal H(\lambda, X)$ remain isolated from one another and are the same in
number and type.  Similarly, the part of the singular graph
$\theta(\mu,X)$ that lies outside a small neighborhood of the support of
$\mu$ will be close (in the $C^1$ topology of embedded graphs) to
the corresponding part of $\theta(\lambda,X)$. 

We emphasize that the constructions above are not
continuous in any {\em neighborhood} of $\mu$ in the measure topology
of $\ML(S)$; rather, maximality of $\mu$ implies continuity at
$\mu$, since for maximal laminations, convergence in measure and
in the Hausdorff sense are the same.

Let us further assume that $\horizontal(\mu,X)$ and $\omega$ satisfy:
\begin{enumerate}
\item[(H1)] The foliation $\horizontal(\mu,X)$ has no saddle connections 
(and in particular, it is minimal). 
\item[(H2)] The horizontal sides of rectangles in $\calR(\omega,\mu, X)$
  containing the endpoints of $\omega$ do not meet the singularities 
  of $\horizontal(\mu,X)$.
\end{enumerate}
Then we can conclude that for $\lambda$ sufficiently close to
$\mu$ and an arc $\omega_\lambda \subset \theta(\lambda,X)$ sufficiently close 
to $\omega$, the rectangle decomposition $\calR(\omega_\lambda,\lambda,X)$ is 
well-defined and is topologically equivalent (i.e.~isotopic) to 
$\calR(\omega,\mu,X)$. First, note that $\omega_\lambda$ is still disjoint from
$\lambda$ and transverse to $\horizontal(\lambda,X)$.
Moreover, the condition (H1) ensures that every point in $P_1$
and $P_2$ corresponds to a unique point in $\omega_\lambda$.
(Note that if $\horizontal(\mu,X)$ had a saddle connection, both saddle
points would project to the same point in $\omega$. But the
corresponding points in $\horizontal(\lambda,X)$ may project to different points 
in $\omega_\lambda$.) Thus the rectangle decompositions 
$\calR(\omega,\mu,X)$ and $\calR(\omega_\lambda,\lambda,X)$ are topologically 
equivalent.

In order to analyze rectangle decompositions for laminations near a
given one, it will be convenient to work with an open neighborhood $U
\subset \ML(S)$ of $\mu$ and to extend the transversal $\omega$
to a family of transversals $\{ \omega_\lambda: \lambda \in U\}$.  
We require that this family satisfy the conditions:
\begin{enumerate}
\item[(T1)] For each $\lambda \in U$, the arc $\omega_\lambda$ lies in
  the singular locus $\theta(\lambda,X)$, and its endpoints
  are disjoint from the vertices of $\theta(\lambda,X)$.

\item[(T2)] The family of transversals is \emph{continuous at
  $\mu$}, meaning that for any $\lambda_n \in U$ such that
  $\lambda_n \to \mu$ in the measure topology, the transversals
  $\omega_{\lambda_n}$ converge to $\omega = \omega_\mu$ in the 
  $C^1$ topology.
\end{enumerate}
Note that for any maximal lamination $\mu \in \ML(S)$, we can
start with a transversal $\omega$ in an edge of its associated singular locus
and construct a family satisfying the conditions above on some neighborhood 
of $\lambda$ in $\ML(S)$.   For example, we can take $\omega_\lambda$ to
be the arc in $\theta(\lambda,X)$ whose endpoints are closest to those of
$\omega$.  The $C^1$ convergence of these arcs as $\lambda \to
\mu$ follows from the convergence of the singular graphs, once we
choose the neighborhood of $\lambda$ in $\ML(S)$ so that the original arc
$\omega_\mu$ has a definite distance from the support of any
lamination in the neighborhood.

\subsection{Geometric stability of the decomposition}
\label{subsec:Geometric}
To quantify the geometry of a rectangle decomposition, rather than its
topology, we introduce parameters describing aspects of the shape of a
rectangle $R \in \calR (\omega,\mu,X)$.
Let $\horizontal=\horizontal(\mu,X)$ and define
$$
\height_R(\mu) =  \text{$\horizontal$--measure of a vertical edge of $R$}.
$$
By construction, the vertical edges of $R$ have the same $\horizontal$-measure, 
so this is well-defined, and is equal to the length of any arc in $R \cap \mu$.

Since we are assuming $\horizontal (\mu, X)$ has no saddle connections,
each horizontal edge of $R$ either contains exactly one singularity 
of $\horizontal$ or does not contain any singularities, but ends at an endpoint of 
$\omega$.  In the former case, the
singularity divides the edge into two \emph{horizontal half-edges}.  Although in the
latter case, the edge is not divided, we nonetheless refer to it as
a horizontal ``half-edge" and include it in
the set $\calI_R$ of horizontal half-edges of $R$.  Define
$$ 
\length_R(\mu) = \max_{I \in \mathcal I_R} \ell(I),
$$ 
where $\ell(I)$ denotes the hyperbolic length of $I$. 

Also define
$$ 
\measure_R(\mu) = \min_{I \in \mathcal I_R} \I(I,\mu),
$$
where $\I(\param,\param)$ is the intersection number with the
transverse measure of $\mu$.

We consider the variation of the rectangle parameters over $U$,
continuing under the assumption that (H1),(H2),(T1),(T2) hold.
By construction, the parameters $\height_R(\mu)$ and
$\length_R(\mu)$ depend continuously on the foliation
$\horizontal(\mu,X)$ and on a compact part of the the singular locus
$\theta(\mu,X)$, both of which vary continuously with the support
of $\mu$ in the Hausdorff topology.  Thus both of these parameters
are continuous at a maximal lamination.  And, $\measure_R(\mu)$ varies
continuously with $\mu$.

For future reference, we summarize this discussion in the following
lemma:
\begin{lemma}
\label{lem:Stability}
Suppose $\mu$ is maximal and $\omega$ is a transversal such that the
pair satisfy (H1) and (H2) above.  Then there is a neighborhood $U$ of
$\mu$ in $\ML(S)$ and a family of transversals satisfying (T1) and (T2)
such that the associated rectangle decompositions are all topologically
equivalent, and  such that for any $\lambda \in U$, we have:
\begin{equation}
\begin{split}
\length_R(\lambda) &< 2\, \length_R(\mu);\\
\measure_R(\lambda) & > \frac{\measure_R(\mu)}{2};\\
\height_R(\lambda) &>  \frac {\height_R(\mu)}2.
\end{split}
\end{equation}
\end{lemma}

\section{Construction of a Quasiconformal Map}
\label{sec:TheMap}
Our plan is to use the rectangle decomposition to define a quasiconformal map 
from the grafting ray to a \Teich ray.  In order to bound the quasiconformal constant, 
however, we need control over the shapes of the rectangles.  Thus we first consider 
a \emph{standard} surface $\Xstd$ for which the rectangle
decomposition is well-behaved.

We will need the following lemma. 
\begin{lemma} \label{lem:Dense} For every maximal lamination
$\mu$, the set $V_\mu$ of Riemann surfaces $Y \in \T(S)$ where 
$\horizontal(\mu,Y)$ has no saddle connections is the intersection of a countable 
number of open dense subsets of $\T(S)$.  For each $Y \in V_\mu$, there
is an arc $\omega$ in $\theta(\mu,Y)$ satisfying (H2) above.  
\end{lemma}

\begin{proof}
Since every point in $\ts$ comes with a marking, we can consider $\mu$ and 
$\horizontal(\mu,X)$ as a measured lamination and a measured foliation on $S$,
respectively.  Since
$\mu$ is maximal, there is one singularity of $\horizontal(\mu,X)$ contained
in each complementary ideal triangle.  Let $\calA=\calA(\mu)$ 
be the set of homotopy classes of arcs connecting the singular points of 
$\horizontal(\mu,X)$.  For any $Y\in \ts$, the set of homotopy classes of arcs 
connecting the singular points of $\horizontal(\mu,Y)$ is identified with $\calA$ via 
the marking map $S \to Y$.

Let $\alpha$ be an arc in $\calA$, and consider the set $V_\alpha$
of points $Y\in \ts$ such that $\alpha$ is not a saddle connection of
$\horizontal(\mu,Y)$. Suppose that $Z$ is in the complement of this set 
$V^c_\alpha$. Let $Z_t$ be the image of $Z$ after applying a left earthquake 
along the measured lamination $t\mu$ and let $\horizontal_t=\horizontal(\mu, Z_t)$. 
An arc in $\calA$ appears as a saddle connection of $\horizontal_t$ if its 
$\horizontal_t$--measure is zero. But the $\horizontal_t$--measure of each arc is 
a linear function of $t$; by the definition of 
the earthquake flow, the $\horizontal_t$--measure of such an arc is equal to its
$\horizontal_0$ measure plus its $\mu$--measure times $t$. Since
$\mu$ is maximal, every arc in $\calA$ has to intersect $\mu$ and hence the 
$\horizontal_t$--measure cannot remain constant. Therefore, $Z_t$ is in 
$V_\alpha$ for every $t>0$.  Since we can apply the same argument for right 
earthquakes, it follows that $V^c_\alpha$ is a closed subset of $\ts$ of co-dimension 
at least one. Thus, $V_\alpha$ is an open dense subset of $\ts$.  Since $\calA$ 
consists of a countable number of elements, the intersection 
$$
V_\mu= \cap_{\alpha \in \calA} V_\alpha
$$ 
is an intersection of a countable number of open dense subset of $\ts$.

For $Y \in V_\mu$, choose an arc $\omega_0  \in \theta(\mu,Y)$.  
There are finitely many leaves of $\horizontal(\mu,Y)$ that contain singularities, 
and these intersect $\omega_0$ in a countable set of points.  Any sub-interval 
$\omega \subset \omega_0$ whose endpoints are in the complement of this set
will satisfy (H2).
\end{proof}

\subsection{The standard surface}
\label{sec:Standard}
Consider a pair of pseudo-Anosov maps $\varphi$ and $\bvp$, 
so that the associated stable lamination $\nu$ and $\bnu$
are distinct and maximal. We perturb $X$ to a Riemann surface $\Xstd$ so that 
both orthogonal foliations $\horizontal=\horizontal(\nu,\Xstd)$ and 
$\bH=\horizontal(\bnu, \Xstd)$ satisfy (H1). This is possible because, 
by~\lemref{lem:Dense}, the intersection of $V_{\nu}$ and $V_{\bnu}$ is still 
dense and hence is non-empty.

Let $\theta$ and $\btheta$ be the singular loci of $\horizontal$ and
$\bH$ respectively. We choose arcs $\so$ and $\bso$ contained in 
$\theta$ and $\btheta$ respectively, satisfying (H2). 
Let $U$ and $\bU$ be open neighborhoods of $\nu$ and $\bnu$ as in 
\lemref{lem:Stability}. By making $U$ and $\bU$ smaller if necessary, we 
can assume that they are disjoint. 
Let $U_{\std} = U \cup \bU$, and for $\lambda \in \Ustd$, let 
$\mathcal R(\lambda)$ denote the rectangle decomposition
$\calR( \omega_\lambda, \lambda, \Xstd)$. 

For $R \in \calR(\nu)$, we know that $\length_R>0$ and 
$\height_R>0$ (this is true for any non-degenerate rectangle). 
An arc contained in a leaf of $\horizontal$ connecting two points
in $\theta$ must intersect $\lambda$. Therefore, such an arc 
connecting $\so$ to itself or to a singular point of $\horizontal$ 
(which also lies in $\theta$) has to have a positive $\nu$ measure. 
Hence, $\measure_R$ is also positive. Similar statements are true for $\bnu$. 
Define
\begin{align}
\label{eqn:lmc}
\Length &=\max \big\{ \length_R : R \in  \calR(\lambda), \ \lambda \in U_{\std} \big\};
\nonumber \\
\Measure &=\min \big\{\measure_R : R \in  \calR(\lambda), \ \lambda \in U_{\std} \big\};\\
\Height &=\min \, \big\{\height_R : R \in  \calR(\lambda), \ \lambda \in U_{\std} \big\}.
\nonumber 
 \end{align}
Then, by \lemref{lem:Stability}, $\Height$,  $\Measure$ and $\Length$ are
finite and positive.  These constants give uniform control over the shapes of all 
rectangles in any rectangle decomposition $\calR(\lambda)$ for $\lambda \in \Ustd$.

 For the rest of this section 
we restrict our attention to laminations 
$\lambda$  in $\Ustd$ only. 
We prove that the grafting ray $\gr(s\lambda, \Xstd)$ fellow 
travels a \Teich geodesic with constants depending on $\Xstd$, $\so$,
$\bso$ and $\Ustd$ but not on $\lambda$ (\propref{prop:TheMap}).

\subsection{Rectangle decomposition of $\gr(s\lambda,\Xstd)$.}

The rectangle decomposition $\calR(\lambda)$ can be extended to a
rectangle decomposition of the grafted surface $\gr(s\lambda,\Xstd)$
that is adapted to its Thurston metric rather than the hyperbolic
metric that uniformizes it. The surface $\gr(s\lambda, \Xstd)$ is obtained by 
cutting $\Xstd$ along the isolated leaves of $\lambda$ and 
attaching a cylinder of the appropriate thickness in their place. That is, the complement 
of the isolated leaves of $\lambda$ in $\Xstd$ is canonically homeomorphic to the 
complement of the corresponding cylinders in the grafted surface. 
However, when $\lambda$ has leaves that are not isolated,
 the complement of the cylinders changes as a metric space.
The length of an arc in the Thurston metric of $\gr(s\lambda, \Xstd)$
disjoint from the cylinders is its hyperbolic length plus its $s\lambda$--measure. 

The rectangle decomposition $\calR(\lambda)$ defines a rectangle decomposition
$\calR_s(\lambda)$ of $\gr(s\lambda, \Xstd)$ as follows. 
A rectangle $R$ in $\calR(\lambda)$ is extended to a rectangle $R_s$ by 
cutting along each isolated arc in $\lambda \cap R$ and inserting a Euclidean rectangle of 
width $s$ times the original $\lambda$--measure carried by the arc, as in Figure
\ref{fig:grafted}. Then $\calR_s(\lambda)$ is the collection of 
rectangles $R_s$. The foliation $\horizontal$ can be extended to a foliation
$\horizontal_s$ of $\gr(s\lambda, \Xstd)$; inside the cylinders corresponding to
isolated leaves, $\horizontal_s$ is the foliation by geodesic arcs (in the Euclidean
metric on the cylinder) that are perpendicular to the boundaries of the cylinder. 

\begin{figure}[tb]
\begin{center}
\includegraphics{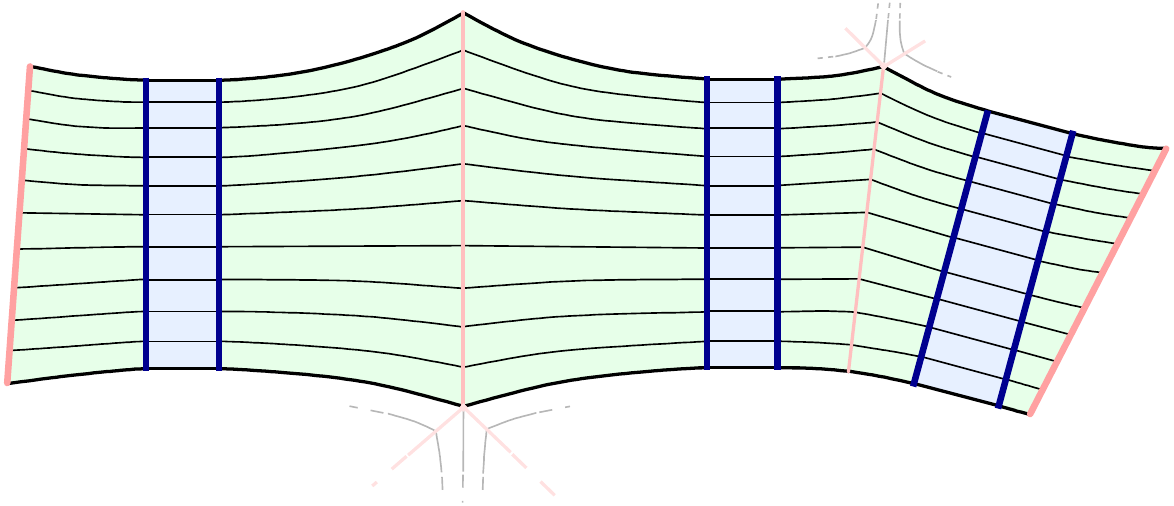}
\caption{The rectangle decomposition extends naturally to the grafted
  surface, replacing isolated arcs of $R \cap \lambda$ with Euclidean rectangles
  contained in the grafted annuli.\label{fig:grafted}}
\end{center}
\end{figure}

\subsection{Foliation parallel to $\lambda \cup \graph$}
\label{sec:coordinates}
Let $R$ be a rectangle in $\calR(\lambda)$.  We foliate $R$ with
geodesic arcs parallel to $\lambda \cup \theta$ as follows.
A component of $R \setminus (\lambda \cup \theta)$ is a geodesic
quadrilateral that has a pair of opposite sides lying in $\theta$ and $\lambda$
respectively. Consider this quadrilateral in the hyperbolic plane, where 
these opposite sides are contained in a pair of disjoint infinite geodesics
$g_1$ and $g_2$. If $g_1$ and $g_2$ do not meet at infinity, the 
region between them can be foliated by geodesics that
are perpendicular to the common perpendicular of $g_1$ and $g_2$. 
If $g_1$ and $g_2$ meet at infinity, the region between them can be foliated by
 geodesics sharing the same endpoint at infinity. 
This foliation restricts to a foliation of the quadrilateral by arcs.
Applying the same construction for each component in each rectangle,
we obtain a foliation $\vertical= \vertical( \omega_\lambda,\lambda, \Xstd)$ of $\Xstd$ that
is transverse to $\horizontal$. Note that, unlike $\horizontal$, the vertical foliation
$\vertical$ does not have a natural transverse measure.
  
Similar to $\horizontal$, the foliation $\vertical$
 can be extended to a foliation  $\vertical_s$ of 
the grafted surface $\gr(s\lambda, \Xstd)$; inside cylinders
corresponding to isolated leaves, $\vertical$ extends as the orthogonal
foliation to $\horizontal_s$. 

\subsection{Projections along $\vertical_s$.}
\label{sec:projections}
\begin{figure}[tb]
\begin{center}
\includegraphics{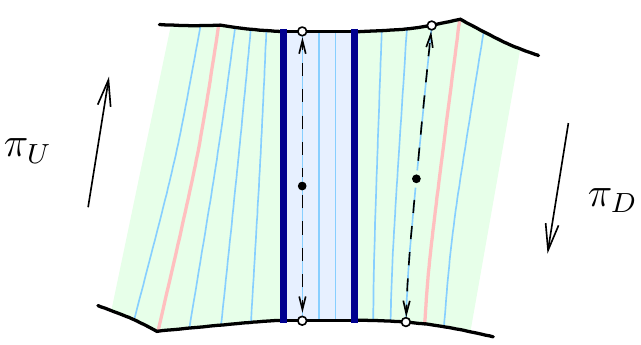}
\caption{A portion of a rectangle $R_s$ foliated by leaves of
  $\vertical_s$ in the case where $\lambda$ is a weighted multi-curve.
  The central rectangle is part of an annulus that has been grafted
  along $\lambda$. The top and bottom edges of $R_s$ are arcs in the
  leaves of $\horizontal_s$.\label{fig:projection}}
\end{center}
\end{figure}
Let $R$ be a rectangle in $\calR(\lambda)$. Orient $R$ so that the notions
of up, down, left and right are defined; these are still well defined 
for the grafted rectangle $R_s \in \calR_s(\lambda)$. 
We assume that the top and the bottom edges
are horizontal and the left and the right edges are vertical. 

Fixing the rectangle $R_s$, we define the map $\down$ from $R_s$
to the bottom edge of $R_s$ to be the projection downward 
along the leaves of $\vertical_s$ and the map $\up$ from $R_s$
to the top edge of $R_s$ to be the projection upward along the leaves of $\vertical_s$
(see Figure~\ref{fig:projection}). Also, define a map $h \from R_s \to \R_+$
to be the height. That is, for $p \in R_s$, $h(p)$ is the $\horizontal_s$--measure of 
any arc (transverse to $\horizontal$) connecting $p$ to the bottom edge of $R_s$. 

\begin{lemma} \label{Lem:bi-Lip}
There is a constant $\LipBound>0$ depending only on $\Xstd$ such that, for 
every $\lambda \in \Ustd$ and $R_s \in \calR_s(\lambda)$ equipped 
with the Thurston metric, the following holds.
\begin{enumerate}
\item The maps $\up$ and $\down$ are $\LipBound$--Lipschitz. 
Furthermore, the restrictions of these maps to a leaf of $\horizontal_s$ are
$\LipBound$--bi-Lipschitz. 
\item The map $h$ is $\LipBound$--Lipschitz. Furthermore, the restriction of $h$ 
to a leaf of $\vertical_s$ is $\LipBound$--bi-Lipschitz. 
\end{enumerate}
\end{lemma}

\begin{proof}
The lemma clearly holds for the interior of added cylinders with
$\LipBound=1$ as $\up$, $\down$ and $h$ are just Euclidean
projections. As mentioned before, in the complement of these added
cylinders, the Thurston length of an arc in $R_s$ is the sum of the
hyperbolic length of this arc and its $s\lambda$--transverse
measure. As one projects an arc up or down, the $s\lambda$--measure
does not change. Therefore, to prove the first part of the lemma, we
need only to prove it for the restriction of $\up$ and $\down$ to
every component of $R \setminus (\lambda \cup \graph)$. Similarly,
proving part two in each component of $R \setminus (\lambda \cup
\graph)$ is also sufficient. This is because showing $h$ is Lipschitz
with respect to the hyperbolic metric implies that it is Lipschitz
with respect to the Thurston metric as well (since the Thurston metric
is pointwise larger \cite[Prop.~2.2]{tanigawa}).  Also, the leaves of
$\vertical$ reside in one component.

Let $Q$ be a component of $R \setminus (\lambda \cup \graph)$.  
We know that $Q$ is a hyperbolic quadrilateral with one vertical edge $e_1$ in 
$\lambda$ and the other $e_2$ in $\theta$. Since $\horizontal$ was defined by
closest point projection to $\lambda$, the top and bottom edges of $Q$ make
an angle of $\pi/2$ with the edge $e_1$. The hyperbolic length of $e_1$ is
bounded by the hyperbolic length of $\omega$ and the maximum distance 
between $e_1$ and $e_2$ is bounded above by the diameter of $\Xstd$.
That is, fixing $\Xstd$ and $\omega$, the space of possible shapes (after including 
the degenerate cases) is compact in Hausdorff topology. For each possible
quadrilateral $Q$, the maps in question are Lipschitz (including the degenerate
cases where the length of $e_1$ is zero, or, $e_1$ and 
$e_2$ coincide) and the Lipschitz constants vary continuously with shape. 
Hence, the maps $\up$, $\down$ and $h$ are uniformly Lipschitz.

Also, the restriction of $\up$ and $\down$ to leaves of
$\horizontal$ are always bijections and have positive derivatives and the restriction
of $h$ to a leaf of $\vertical$ is a bijection and has a positive derivative. Thus, there 
is a uniform lower bound for these derivatives and hence they are uniformly
 bi-Lipschitz maps. 
\end{proof}

\subsection{Mapping to a singular Euclidean surface}
For each rectangle $R_s$ in $\calR_s(\lambda)$ consider a corresponding 
Euclidean rectangle $E_s$ with width  equal to the $s\lambda$--measure 
of the horizontal edges of $R_s$ and height equal to the $\horizontal_s$--measure 
of the vertical edge. Recall that a horizontal edge of a rectangle $R_s$ is divided
into two (or one; see \secref{subsec:Geometric})
 horizontal half-edges; the set of horizontal 
half-edges of $R_s$ is denoted by $\calI_{R_s}$. We also mark a special point
on each the horizontal edge of $E_s$, dividing it into horizontal half-edges, so that 
the Euclidean length of the interval associated to $I \in \calI_{R_s}$ is equal to 
$\I(s\lambda, I)$. 

We fix a correspondence between horizontal half-edges of $R_s$ and $E_s$. 
Glue the rectangles $E_s$ along these horizontal half-edges and the vertical 
edges with Euclidean isometries in the same pattern as the rectangles $R_s$ 
are glued in $\calR_s(\lambda)$. Each horizontal half-edge $I$ appears in two 
rectangles, but $\I(s\lambda, I)$ is independent of which rectangle we choose. 
Similarly, the vertical edges appear in two rectangles each, but their 
$\horizontal_s$--measures 
are independent of the choice of the rectangle. Hence the lengths of corresponding 
intervals in $E_s$ match and the gluing is possible; it results in a singular 
Euclidean surface $\calE_s$. Our goal in this section is to define a quasiconformal 
map $F$ between $\gr(s\lambda, \Xstd)$ and $\calE_s$. 

Consider the rectangles in $\calR_s(\lambda)$ as sitting in $\gr(s\lambda, \Xstd)$ 
and the corresponding rectangles as sitting in $\calE_s$. The horizontal 
half-edges and the vertical edges of the rectangles
 form a graph in $\gr(s\lambda, \Xstd)$. First, we define a map $f_s$ from 
this graph to the associated graph in $\calE_s$. Note that, for the gluing
to work, the map should depend on the edge only and not on the choice of 
rectangle containing it. We map any horizontal half-edge $I \in \calI_{R_s}$ of the rectangle $R_s$ linearly onto the associated interval in $E_s$, where we take 
$I$ to be equipped with the induced Thurston metric.  We map a vertical edge 
$J$ to the associated vertical edge in $E_s$ so that the 
$\horizontal_s$--measure is preserved. Note that $f_s$
is distance non-increasing.

Now that the map $f_s$ is defined on the $1$--skeleton, we extend it to a map 
$$F_s \from \gr(s\lambda, \Xstd) \to \calE_s,$$ 
as follows: for each rectangle $R_s$, we send leaves of $\vertical$ to geodesic
segments in $E_s$ so that $\horizontal_s$ is preserved. More precisely, 
for every point $p \in R_s$, consider the points 
$q=f_s(\up(p))$ and $q'=f_s(\down(p))$. Then, let $F_s(p)$ be the point in the
segment $[q,q']$ whose distance from the bottom edge of $E_s$ is $h(p)$. 
We observe the following:

\begin{lemma} \label{lem:Slope}
The slope of the segment $[q,q']$ is uniformly bounded below. 
\end{lemma}
\begin{proof}
Consider the rectangle $E_s$ in $\R^2$ with the horizontal and the vertical
edges parallel to the $x$--axis and the $y$--axis respectively and the
bottom left vertex at the origin. 
The height of $R_s$ is at least $\Height$, so the same is true of 
$E_s$. Hence, we need to show that
the $x$--coordinates of $q$ and $q'$ differ by at most a bounded amount. 

Let $I$ be the interval connecting the top left vertex of $R_s$ to $\up(p)$
and let $I'$ be the interval connecting the bottom left vertex of $R_s$ to 
$\down(p)$. Note that, $\I(s\lambda, I) = \I(s\lambda, I')$. Also, the
Thurston length of $I$ is equal to $\I(s\lambda, I)$ plus the hyperbolic
length of $I$, which is bounded above by $\Length$ and the same holds
for $I'$. Therefore, the Thurston lengths of $I$ and $I'$ differ by at most
$2\Length$. It remains to be shown  that the difference between the
$x$--coordinate of $q$ and the Thurston length of $I$ and
the difference between the $x$--coordinate of $q'$ and the Thurston length of $I'$ 
are uniformly bounded. 

To see this last assertion note that, as $p$ moves to the right along a leaf of 
$\horizontal$, the difference between the Thurston length of $I$ and the
$x$--coordinate of $q$ increases (the derivatives of the map $f_s$ are 
always less than $1$). Thus, this difference is an increasing function 
that varies from zero to $\length_R$ and hence is uniformly
bounded by $\Length$. A similar argument works for $q'$ and $I'$. 
This finishes the proof. 
\end{proof}

\subsection{Bounding the quasiconformal constant}
We now bound the quasiconformal constant of $F_s$.  First we introduce some
notation.  

Given two quantities $A$ and $B$, we say $A$ is comparable to $B$
and write  $A \emul B$, if
$$ \frac{1}{c} B \le A \le cB,$$
for a constant $c$ that depends on predetermined values, 
such as the topology of $S$, or $\Height, \Length, \Measure$ 
as defined above. Similarly, $A \eadd B$ means that there is a 
constant $c>0$ such that
$$ B - c < A < B + c,$$
where the constant $c$ may have similar dependencies.
We say, $A$ is of order of $B$ and write $A \lmul B$ if 
$A \leq cB$, for $c$ as above. The notation $A \ladd B$ is defined analogously. 
 \begin{proposition} 
\label{prop:qc-calculation}
Let $s_0 > 0$.  Then there is a constant $k$, depending on $s_0$, such that
for any $\lambda \in \Ustd$ and $s \geq s_0$,
the map 
$$
F_s \from \gr(s\lambda, \Xstd) \to \calE_s
$$ 
is $k$--quasiconformal.   
\end{proposition}

\begin{proof}
Let $p_1$ and $p_2$ be two points in $R_s$. We will show:
$$
d_{Th}(p_1, p_2) \emul d_{\R^2} \big(F_s(p_1), F_s(p_2)\big). 
$$
Let $\ep=d_{Th}(p_1, p_2)$. By \lemref{Lem:bi-Lip}, we have
$$
d_{Th}(\up(p_1), \up(p_2)) \lmul \ep, \quad
d_{Th}(\down(p_1), \down(p_2)) \lmul \ep,
$$
and
$$
|h(p_1) -h(p_2) | \lmul \ep.
$$
For $i=1,2$, let $q_i = f_s(\up(p_i))$ and $q_i'=f_s(\down(p_i))$.   Note that the points $q_1,q_2$ lie on a horizontal line, the top side of the rectangle $E_s$, and similarly $q_1',q_2'$ lie on the bottom of the rectangle.
Since, $f_s$ is distance non-increasing, we have
$$
d_{\R^2}(q_1, q_2) \lmul \ep 
\quad\text{and}\quad
d_{\R^2}(q'_1, q'_2) \lmul \ep.
$$
Hence, the horizontal distance between the lines $[q_1, q_1']$ and $[q_2, q_2']$
is of order of $\ep$ at every height. Also, by \lemref{lem:Slope}, these lines have
slope bounded below.   For $i=1,2$, the point $F_s(p_i)$ lies on the line $[q_i,q_i']$ at height $h(p_i)$, so if we cut the lines $[q_1,q_2']$ and $[q_2,q_2']$ by the pair of horizontal lines corresponding to these heights, then $F_s(p_1)$ and $F_s(p_2)$ are opposite corners of the resulting quadrilateral.

A quadrilateral which has two opposite horizontal sides of length of order $\ep$,
a height of order $\ep$ and a definite angle at each vertex (guaranteed here 
by the slope condition) has a diameter that is also of order of $\epsilon$, 
thus the distance between $F_s(p_1)$ and $F_s(p_2)$
is also of order of $\ep$.

In the other direction, suppose $\delta =  d_{\R^2} \big(F_s(p_1), F_s(p_2)\big)$.
The restriction of $F_s$ to a horizontal half-edge $I$ is linear,
with derivative equal to $\I(s\lambda, I)$ divided by the Thurston
length of $I$, and we have
$$ 
\frac{ s\Measure}{ \Length + s\Measure} 
  \geq \frac{s_0 \Measure}{\Length + s_0 \Measure},
$$
so this derivative is bounded below independent of $s$.  The corresponding upper bound for the derivative of the inverse map gives
$$
d_{Th}(\up(p_1), \up(p_2)) \lmul \delta,
$$
and by \lemref{Lem:bi-Lip}, 
$$
|h(p_1)-h(p_2)| \lmul \delta.
$$
Consider the leaf $l$ of $\horizontal$ passing through $p_1$ and let $p$
be the intersection of $l$ with the leaf of $\vertical$ passing through $p_2$. 
Since $h$ restricted to a leaf of $\vertical$ is uniformly bi-Lipschitz, 
the arc  $[p_2,p]$ has a length of order $\delta$. Also, since $\up$ restricted to a 
leaf of $\horizontal$ is uniformly bi-Lipschitz, the arc $[p,p_1]$ along a leaf of 
$\horizontal$ has a length of order $\delta$. The triangle inequality implies: 
\begin{equation*} 
d_{\text{Th}}(p_1, p_2) \lmul \delta. 
\end{equation*}
The same argument works for every rectangle in $\calR_s(\lambda)$. That is,
the map $F$ is uniformly bi-Lipschitz and thus uniformly quasiconformal. 
\end{proof}

\subsection{The \Teich ray}

The map $F_s$ provides a marking for $\calE_s$ and thus we can consider
$\calE_s$ as a point in $\ts$. We will show that after 
reparameterization, this family of points traces a \Teich geodesic ray. 
The surface $\calE_s$ defines a quadratic differential $q_s$:
locally away from the singularities,  $\calE_s$ can be identified with a 
subset of $\R^2$ sending the horizontal and the vertical foliations to lines
parallel to the $x$--axis and $y$--axis respectively. We define $q_s$ in this local
coordinate to be $dz^2$. 

The leaves of the horizontal lines on $E_s$ (defined locally by $y$=constant)
match along the gluing intervals to define a singular foliation on $\calE_s$. 
Then $|dy|$ defines a transverse measure on this foliation. 
From the construction, we see that this measured foliation represents
$\horizontal_s$.  Similarly, the vertical lines on $E_s$ define a singular foliation with
transverse measure $|dx|$ to give 
a measured foliation representing $s\lambda$. 

The Euclidean area of $\calE_s$ is $s \cdot \ell_{\Xstd}(\lambda)$.
Thus, scaling by $1/{(s \ell _{\Xstd}(\lambda))}$, we get a unit-area quadratic
differential $q_s$ on $\calE_s$ whose vertical
and horizontal foliations are respectively,
$\sqrt{s} \lambda/\sqrt{\ell_{\Xstd}(\lambda)}$ and $\horizontal/\sqrt{s \ell _{\Xstd}(\lambda)}$.

Letting $s = e^{2t}$,
 we obtain
the one-parameter family of quadratic differentials 
$$q_s = \left[ \begin{matrix} e^t & 0 \\ 0 & e^{-t}  \end{matrix} \right] q_{1}.$$
It is well known that the underlying conformal structures of these quadratic differentials
traces a \Teich geodesic parameterized by arc length with parameter $t$.

We summarize the discussion in the following:
\begin{proposition}
\label{prop:TheMap} Let $t_0 \in \R$. Then there is a constant $\quasi$, depending on
$t_0$ such that the following holds:
For any $\lambda \in  \Ustd$ 
there is a Riemann surface $Y_\lambda$ such that for all $t \geq t_0$
$$
d_\T \Big(\gr(e^{2t} \lambda, \Xstd),\G(t,\lambda,Y_\lambda) \Big) \leq \quasi.
$$
\end{proposition}

This proposition is the final result in our study of the orthogonal
foliation and rectangle decomposition of a grafted surface, and it
provides the basic relation between Teichm\"uller geodesic rays and
grafting rays in the proof of the main theorem.

Before proceeding with this proof, however, we need to estimate the
effect (in terms of the Teichm\"uller metric) of moving the base point
of a grafting ray.  This is addressed in the next two sections, where
we discuss projective grafting and its extension to quasi-Fuchsian
groups, leading to the proofs of Theorems \ref{thm:extension} and
\ref{thm:Lipschitz}.  In \secref{sec:Conclusion}, these results will
be combined with \propref{prop:TheMap} in order to prove
\thmref{thm:Main} for surfaces without punctures.

\section{Projective structures, grafting, and bending}
\label{sec:projective}

In this section we collect some background material on complex
projective structures, grafting, and bending.  We also establish some
basic compactness results for projective structures and their
developing maps, which will be used in the proof of \thmref{thm:extension}.
We emphasize that in \secsref{sec:projective}{sec:Lipschitz} the argument 
does not depend on whether or not $S$ has punctures.

\subsection{Deformation space}

Let $\P(S)$ denote the deformation space of marked complex projective
structures on the surface $S$.  Each such structure is defined by an
atlas of charts with values in $\CP^1$ and transition functions in
$\PSL_2\C$.  If $S$ has punctures, we also require that a neighborhood
of each puncture is projectively isomorphic to a neighborhood of a
puncture in a finite-area hyperbolic surface (considered as a
projective structure, using a Poincar\'e conformal model of $\hyp^2$).
For background on projective structures, see \cite{gunning},
\cite{kamishima-tan}, \cite{tanigawa} \cite{survey}.

There is a forgetful projection map $\pi \from \P(S) \to \T(S)$, which
gives $\P(S)$ the structure of a complex affine vector bundle modeled
on the bundle $\qd(S) \simeq T^{1,0}\T(S)$ of integrable holomorphic quadratic
differentials.  We denote the fiber of $\qd(S)$ over $X \in \T(S)$ by
$Q(X)$.  A projective structure $Z$ with developing map $f \from
\Tilde{Z} \to \CP^1$ is identified with the quadratic differential
$S(f)$ on $\pi(Z)$, where
$$ S(f) = \left ( \frac{f''}{f'} \right )' - \frac{1}{2} \left (
  \frac{f''}{f'} \right )^2$$
is the Schwarzian derivative.  This \emph{Schwarzian parameterization}
gives $\P(S)$ the structure of a complex manifold of complex dimension
$3 | \chi(S) | = 2 \dim_\C \T(S)$.

\subsection{Projective grafting and holonomy}

A grafted Riemann surface carries a natural projective structure. A
local model for this projective grafting construction in the universal
cover of a hyperbolic surface $X$ is given by cutting the upper
half-plane $\hyp$ along $i \R^+$ and inserting a sector of angle $t$.
The quotient of this construction by a dilation $z \mapsto e^\ell z$
corresponds to inserting a cylinder of length $t$ and circumference
$\ell$ along the core geodesic $\gamma$ of a hyperbolic cylinder.
Using this model to define projective charts on a grafted surface
$\gr(t \gamma,X)$ gives a complex projective structure $\Gr(t\gamma,X)
\in \P(S)$.  As with grafting of complex structures, there is an
extension of this \emph{projective grafting map} to
$$ \Gr \from \ML(S) \times \T(S) \to \P(S),$$
which satisfies $\pi \circ \Gr = \gr$, i.e.~the underlying Riemann
surface of the projective structure $\Gr(\lambda,X)$ is
$\gr(\lambda,X)$.

For projective grafting with small weight along a simple closed curve,
the developed image of $\Tilde{\Gr(t \gamma,X)}$ is an open subset of
$\Hat{\C}$ obtained from $\Tilde{X} \simeq \hyp$ by inserting
\emph{$t$--lunes} along the geodesic lifts of $\gamma$, adjusting the
complementary regions in $\hyp^2$ by M\"obius transformations so that
these lunes and hyperbolic regions fit together.  The picture for
larger $t$ is locally similar, but on a larger scale the developing
map may fail to be injective.

Let $\X(S)$ denote the $\PSL_2(\C)$--character variety of $\pi_1(S)$,
that is,
$$ \X(S) = \Hom(\pi_1(S), \PSL_2(\C)) \sslash \PSL_2(\C), $$
where $\PSL_2(\C)$ acts on the variety $\Hom(\pi_1(S), \PSL_2(\C))$ by
conjugation of representations and $\sslash$ denotes the quotient
algebraic variety in the sense of geometric invariant theory (see
\cite{heusener-porti}, \cite[\S II.4]{morgan-shalen}).  Let $\hol \from
\P(S) \to \X(S)$ be the holonomy map, assigning to each projective
structure $Z$ the holonomy representation $\hol(Z) \from \pi_1(S) \to
\PSL_2(\C)$, well-defined up to conjugacy, which records the
obstruction to extending projective charts of $Z$ along homotopically
nontrivial loops.

\subsection{Bending}
\label{subsec:bending}
The composition of the grafting and holonomy maps,
\begin{equation*}
B = \hol \circ \Gr \from \ML(S) \times \T(S) \to \X(S)
\end{equation*}
is the \emph{bending map} (or \emph{bending holonomy map} in
\cite[\S2]{mcmullen}, a special case of the \emph{quakebend} of
\cite[Ch.~3]{epstein-marden}).  If $X \in \T(S)$ corresponds to a
Fuchsian representation $\rho_X \from \pi_1(S) \to \PSL_2\R$ preserving a
totally geodesic plane $\hyp^2 \subset \hyp^3$, then $B(\lambda,X)$ is
a deformation of this representation which preserves a pleated plane
$\pl \from \hyp^2 \to \hyp^3$ obtained by bending $\hyp^2 \simeq
\Tilde{X}$ along the lift of $\lambda$.

For later use, we describe this pleated plane explicitly in terms of
the data $(X,\lambda)$, first in the case where $\lambda = t \gamma$
is supported on a simple closed curve; see \cite{epstein-marden} for
further details.  Lift the closed geodesic $\gamma \subset X$ to a
family $\Tilde{\gamma}$ of complete hyperbolic geodesics in $\hyp^2$.
Given $(x,y) \in (\hyp^2 \setminus \Tilde{\gamma}) \times (\hyp^2
\setminus \Tilde{\gamma})$, let $\{g_1, \ldots g_n\} \subset
\Tilde{\gamma}$ denote the set of lifts of $\gamma$ that intersect the
hyperbolic geodesic segment $\overline{xy}$, ordered so that $g_1$ is
closest to $x$.  Let $(p_i, q_i) \in \Hat{\R} \times \Hat{\R}$ denote the ideal
endpoints of $g_i$, with $p_i$ chosen so as to lie to the left of the
oriented segment $\vec{xy}$.  Define the \emph{bending cocycle}
$$\beta_{\lambda,X}(x,y) = E_t(p_1,q_1) \cdots E_t(p_n,q_n)$$
where $E_\theta(p,q) \in \PSL_2(\C)$ is the elliptic M\"obius
transformation with fixed points $p,q$, rotating counter-clockwise
angle $\theta$ about $p$.  If $x$ and $y$ are contained in the same
component of $(\hyp^2 \setminus \Tilde{\gamma})$, then we define
$\beta_{\lambda,X}(x,y) = \identity$.

Thus the map $\beta_{\lambda,X} \from (\hyp^2 \setminus \Tilde{\gamma})
\times (\hyp^2 \setminus \Tilde{\gamma}) \to \PSL_2(\C)$ is locally constant in
each variable, with a discontinuity along each lift of $\gamma$, where
the values on either side of $g \in \Tilde{\gamma}$ differ by an
elliptic M\"obius transformation fixing the endpoints of $g$.  The
bending cocycle is related to the bending map $B$ as follows: choose a
basepoint $\base \in (\hyp^2 \setminus \Tilde{\gamma})$ and for each
$\alpha \in \pi_1(S)$ define
$$\rho_{X,\lambda}(\alpha) = \beta_{\lambda,X}(\base, \rho_X(\alpha)
\base) \cdot \rho_X(\alpha).$$
Then $B(\lambda,X)$ and $\rho_{X,\lambda}$ are conjugate, i.e.~they
represent the same point in $\X(S)$.

The developing map of $\Gr(t \gamma,X)$ has a similar description in
terms of the bending cocycle:  We define $f \from (\hyp^2 \setminus
\Tilde{\gamma}) \to \Hat{\C}$ by
$$f(y) = \beta_{\lambda,X}(\base, y) \cdot y$$
where in this formula, the M\"obius map $\beta_{\lambda,X}(\base,y)$
  acts on the upper half-plane $\hyp^2$ (and thus on $y$) by the usual
  linear fractional transformation.  Unlike the pleating map, the map
  $f$ is discontinuous along each lift of $\gamma$, where it omits a
  $t$--lune in $\Hat{\C}$.  The developing map of $\Gr(\lambda,X)$
  fills in these lunes with developing maps for the projective annulus
  $\gamma \times [0,t]$.

  The bending map $B$, bending cocycle $\beta$, and the above
  description of the developing map all extend to the case of a general
  measured lamination $\lambda \in \ML(S)$.

\subsection{Quasi-Fuchsian bending}

Let $\QF(S) \subset \X(S)$ denote the quasi-Fuchsian
space of $S$, consisting of conjugacy classes of faithful
quasi-Fuchsian representations of $\pi_1(S)$.  By the Bers
simultaneous uniformization theorem, we have a biholomorphic
parameterization
$$ Q \from \T(S) \times \bar{\T(S)} \to \QF(S). $$
With respect to this parameterization, the space $\F(S) \subset
\QF(S)$ of Fuchsian groups is exactly the diagonal $\{ Q(X,\bar{X}) \:
| \: X \in \ts \}$, and this is a properly embedded totally real
submanifold of maximal dimension.

By the uniformization theorem, we can identity the \Teich space with
the space of Fuchsian groups, $\T(S) \simeq \F(S)$.  We use this
identification to regard projective grafting and bending as maps
defined on $\ML(S) \times \F(S)$.  Kourouniotis
\cite{kourouniotis:bending} showed that the bending map extends
naturally to a continuous map
$$ B \from \ML(S) \times \QF(S) \to \X(S)$$
which is holomorphic in the second factor (since it is the flow of a
holomorphic vector field \cite[Thm.~3]{kourouniotis:continuity}, see
also \cite[\S4]{goldman:complex-symplectic}).

\subsection{Developing maps and compactness}

In the next section, we will need a compactness criterion for sets of
complex projective structures.  Let $Z \in \P(S)$ be a marked complex
projective structure.  An open set $U \subset Z$ \emph{develops
  injectively} if the developing map of $Z$ is injective on any lift
of $U$ to the universal cover.

By the uniformization theorem, the marked complex structure $\pi(Z)
\in \T(S)$ underlying a projective structure $Z \in \P(S)$ is
compatible with a unique hyperbolic metric on $Z$ up to isotopy.  We
say that $Z$ has an \emph{injective $r$--disk} if there is an open disk
in $Z$ of radius $r$ (with respect to this hyperbolic metric) that
develops injectively.  Note that we assume here that $r$ is less than
the hyperbolic injectivity radius of $Z$.

The following lemma is essentially an adaptation of Nehari's estimate
for univalent functions \cite{nehari}:

\begin{lemma}
\label{lem:compactness}
For any compact set $K \subset \T(S)$ and any $\delta > 0$, the set of
projective structures $Z \in \pi^{-1}(K) \subset \P(S)$ that contain
an injective $\delta$--disk is compact.
\end{lemma}

\begin{proof}
  The set of such projective structures is closed, so we need only
  show that it is contained in a compact subset of $\P(S)$.

  Because $K \subset \T(S)$ is compact, the integrable quadratic differentials on
  the Riemann surfaces in $K$ have a definite amount of mass in each
  $\delta$--disk, i.e.~there exists a constant $m(K,\delta) > 0$ such
  that
\begin{equation}
\label{eqn:l1bound}
 \| \phi \|_{L^1(D)} \geq m(K,\delta) \| \phi\|_{L^1(X)}
\end{equation}
for all $X \in K$, $\phi \in Q(X)$, and any open disk $D \subset X$ of
hyperbolic radius $\delta$.  Here $\|\cdot\|_{L^1}$ is the conformally
natural norm on quadratic differentials:
$$ \| \phi \|_{L^1(U)} = \int_U |\phi|.$$

By Nehari's theorem, if $Z$ is a projective structure with Schwarzian
differential $\phi$, then on any open set $U \subset Z$ that develops
injectively, we have
\begin{equation}
\label{eqn:nehari}
| \phi | \leq \frac{3}{2} \rho_U
\end{equation}
where $\rho_U$ is the area element of the Poincar\'e metric of $U$.  

Now suppose that $Z$ contains an injective $\delta$--disk $D$, and let
$\frac{1}{2}D$ denote the concentric disk with radius $\delta/2$ with
respect to the hyperbolic metric of $Z$.  Applying \eqref{eqn:l1bound}
to $\frac{1}{2}D$ and \eqref{eqn:nehari} to $D$, we have
\begin{equation}
\label{eqn:uniform-norm-bound}
\begin{split}
 \|\phi\|_{L^1(X)} &\leq \frac{1}{m(K,\delta/2)} \| \phi
 \|_{L^1(\frac{1}{2}D)}\\
& \leq
\frac{1}{m(K,\delta/2)} \int_{\frac{1}{2}D}  \frac{3}{2} \rho_D
= \frac{3 \area(\frac{1}{2}D, \rho_D)}{2 m(K,\delta/2)}
\end{split}
\end{equation}
where we use the notation $\area(\Omega,\rho)$ for the area of a set
$\Omega$ with respect to the area form $\rho$.  The quantity
$A(\delta) = \area(\frac{1}{2}D, \rho_D)$ depends only on $\delta$,
and using elementary hyperbolic geometry we find
$$ A(\delta) = \frac{2 \pi (1 + \cosh(\delta))}{1 + 2 \cosh \left (
    \frac{\delta}{2} \right )}.$$
Therefore, \eqref{eqn:uniform-norm-bound} gives a uniform upper bound
$N(\delta,K) = 3 A(\delta) / (2 m(K,\delta/2))$ on the
$L^1$ norm of $\phi$.

Since $\|\param\|_{L^1(X)}$ gives a continuously varying norm on the
fibers of $\qd(S) \simeq \P(S)$, the union of the closed
$N(\delta,K)$--balls over the compact set $K \subset \T(S)$ is
compact.
\end{proof}

\subsection{Hyperbolic geometry of grafting}
\label{sec:hyperbolic}

By construction, the grafted surface $\gr(t \gamma,X)$ has an open
subset that is naturally identified with $(X \setminus \gamma)$.  In
this subsection with study the geometry of this open set with respect
to the hyperbolic metric of $\gr(t \gamma,X)$.  We do so by comparing
the hyperbolic metric on a grafted surface with the \emph{Thurston
  metric}, which is obtained by gluing the Euclidean metric of $\gamma
\times [0,t]$ to the hyperbolic metric of $X$ (see \cite[\S2.1]{tanigawa}).

Since the Thurston metric is conformally equivalent to the hyperbolic
metric on $\gr(t \gamma,X)$, its length element can be expressed as
$\rhoth = e^u \rhohyp$, where $\rhohyp$ is the hyperbolic
length element and $u$ is a real-valued function.  The
Gaussian curvature of the Thurston metric is well-defined except on
the boundary of the grafting cylinder, and wherever defined it is
equal to $0$ or $-1$.  These bounds on the curvature correspond to the
density function $u$ weakly satisfying
\begin{equation}
\label{eqn:curvature}
-1 \leq \hyplap u \leq -1 + e^{2u}
\end{equation}
where $\hyplap = \frac{4}{\rhohyp}\partial\bar{\partial}$ is the
Laplace-Beltrami operator of the hyperbolic metric on $\gr(t\gamma,X)$
(see \cite{huber}).  Approximating $u$ by a smooth function and
considering its Laplacian at a minimum, it follows easily from the
right hand inequality of \eqref{eqn:curvature} that $u \geq 0$,
i.e.~the Thurston metric is pointwise larger than the hyperbolic
metric.  Thus $u \leq e^u$, and since the area of the Thurston metric
is $A = \int_{\gr(t \gamma,X)} e^{2u} \rhoth^2 = \left ( \| e^u
\|_2\right )^2$, we have
\begin{equation}
\label{eqn:area}
\| u \|_2 \leq \| e^u \|_2 \leq A^{\frac{1}{2}}.
\end{equation}

These area and curvature considerations are sufficient to give a
pointwise upper bound for $u$:

\begin{lemma}
\label{lem:boundabove}
If $Y$ is a compact hyperbolic surface and $u : Y \to \R^+$ is a
function satisfying inequalities \eqref{eqn:curvature} and \eqref{eqn:area},
then we have
$$ \sup u \leq M(A,Y)$$
where $M$ is a continuous function of $A \in \R$ and of the image of
$Y$ in moduli space.
\end{lemma}

An equivalent geometric statement of this lemma is: \emph{On any
  compact subset of moduli space, the conformal metrics with bounded
  area and with curvatures pinched in $[-1,0]$ are uniformly bounded
  relative to the associated hyperbolic metrics.}

\begin{proof}
  For any compact Riemannian manifold $(M,g)$ and any nonnegative
  function $u$ with square-integrable derivatives satisfying $\Delta_g
  u \geq f$, we have the Di~Giorgi-Nash-Moser maximum principle (see
  e.g.~\cite[\S4.2]{han-lin} \cite[Thm.~9.20]{gilbarg-trudinger}):
\begin{equation}
\label{eqn:moser}
 \sup u \leq C(M,g) \left ( \|u\|_2 + \|f\|_2 \right ),
\end{equation}
where $C(M,g)$ is a constant that for fixed $M$ can be taken to vary
continuously with $g$.

We apply this to the function $u$ in the statement of the lemma, using
the hyperbolic metric of $Y$, so $C(M,g) = C(Y)$.  By
\eqref{eqn:curvature} we can use $f \equiv -1$ and
$$ \left (\| f \|_2\right)^2 = \area(Y,\rhohyp) = 2 \pi |\chi(Y)|.$$
Substituting this and \eqref{eqn:area} into the right hand side of
\eqref{eqn:moser} gives
$$ \sup u \leq C(Y) \left ( A^\frac{1}{2} + \left (2 \pi
    |\chi(Y)|\right)^{\frac{1}{2}} \right ) = M(A,Y).$$
\end{proof}

\noindent The following corollary relates \lemref{lem:boundabove} to the
geometry of grafting.

\begin{corollary}
\label{cor:injectiveball}
  There exists a continuous positive function $r : \ML(S) \times \T(S)
  \to \R^+$ such that for any $(t \gamma,X) \in \ML(S) \times \T(S)$,
  the image of $(X \setminus \gamma)$ in $\gr(t \gamma,X)$ contains a
  ball of radius $r(t \gamma,X)$ with respect to the hyperbolic metric
  on $\gr(t \gamma,X)$.
\end{corollary}

\begin{proof}
  An equivalent statement is that there exists a point in $\gr(t
  \gamma,X)$ whose distance from the grafting cylinder $A$ is at least
  $r(t \gamma,X)$.

  Let $x_0 \in (X \setminus \gamma)$ be a point whose distance from
  $\gamma$ in the hyperbolic metric of $X$ is at least $r_0$, where
  $r_0$ is the radius of the inscribed disk of a hyperbolic ideal
  triangle.  Such a point always exists, since every complete
  hyperbolic surface with geodesic boundary contains an isometrically
  embedded ideal triangle.  We identify $x_0$ with its image in
  $\gr(t \gamma,X)$, which is a point whose Thurston distance from the
  grafting cylinder $A$ is at least $r_0$.

  We will show that the distance $\dhyp(x_0,A)$ from $x_0$ to $A$ with
  respect to the hyperbolic metric on $\gr(t \gamma, X)$ is also bounded below.

  The Thurston metric on $\gr(t \gamma,X)$ is obtained by gluing a
  hyperbolic surface of area $2 \pi |\chi(S)|$ and a cylinder of area
  $t \ell_X(\gamma)$.  Therefore we have
$$ \area(X,\rhoth) = 2 \pi |\chi(S)| + t \ell_X(\gamma) =: A(t \gamma,X).$$
By Lemma \ref{lem:boundabove}, this area bound and the curvature of the
Thurston metric imply that
$$ \sup e^u \leq \exp(M(A(t \gamma,X),\gr(t \gamma,X))).$$
Since the length function is continuous on $\ML(S) \times \T(S)$, and
$\gr : \ML(S) \times \T(S) \to \T(S)$ is continuous, the right hand
side of this estimate is a continuous function $f(t \gamma,X)$.

Let $\alpha$ be a minimizing geodesic arc from $x_0$ to $\partial A$
with respect to the hyperbolic metric of $\gr(t \gamma,X)$.  Then the
length of this arc with respect to the Thurston metric is
$$ \int_\alpha e^u \rhohyp \leq (\sup e^u) \int_\alpha \rhohyp \leq
f(t \gamma,X) \dhyp(x_0, A).$$
Since $x_0$ was chosen so that this length is at least $r_0$, we
have shown that
\begin{equation*}
\dhyp(x_0,A) \geq \frac{r_0}{f(t \gamma,X)}. \qedhere
\end{equation*}
\end{proof}

\section{Grafting is Lipschitz}  \label{sec:Lipschitz}

In this section we prove Theorems~\ref{thm:extension} and
\ref{thm:Lipschitz}, after developing some preliminary results about
quasidisks and quasi-Fuchsian groups.

\subsection{Extension theorem}
\label{subsec:extension}
Let $\QF_\delta(S) \subset \QF(S)$ denote the set of
\emph{$\delta$--almost-Fuchsian groups}, i.e.~quasi-Fuchsian
groups of the form $Q(X,Y)$ where $d_{\T}(X,Y) < \delta$.  Thus
$\QF_\delta(S)$ is a connected, contractible and open neighborhood of
$\F(S)$ in $\QF(S)$.

Let $d_{\QF}$ denote the Kobayashi distance function on 
$$
\QF(S) \simeq\T(S) \times \bar{\T(S)}.
$$  T
he Kobayashi metric on $\T(S)$ (or
$\bar{\T(S)}$) is equal to the Teichm\"uller metric, and the Kobayashi
metric on a product of manifolds is the sup-metric.  Thus we can also
describe $\QF_\delta(S)$ in terms of $d_{\QF}$:
$$ 
\QF_\delta(S) = \big\{ \rho \in \QF(S) \: | \: d_{\QF}(\rho,\F(S)) 
  < \delta \big\}.
$$

As explained in the introduction, our first goal in this section is to
show that 
$$
\Gr\from \ML(S) \times \F(S) \to \P(S)
$$ 
extends holomorphically to $\QF_\Extension(S)$.  The local existence of 
an extension $\Gr(\lambda,\param)$ near $\F(S)$ is clear: The map 
$\Gr_\lambda\from \F(S) \to \T(S)$ is known to be real-analytic (see
\cite{scannell-wolf}\cite{mcmullen}), so it has a holomorphic
extension in a small neighborhood of the totally real manifold $\F(S)
\subset \QF(S)$.  Alternatively, the map $\hol : \P(S) \to \X(S)$ is a
local biholomorphism, so the quasi-Fuchsian bending map $B_\lambda:
\QF(S) \to \X(S)$ constructed by Kourouniotis (see
\cite{kourouniotis:bending}) can be locally lifted through $\hol$ to
define $\Gr(\lambda,\param)$.

Unfortunately it seems difficult to control the domain of definition
of the extensions that arise from these considerations. If using
real-analyticity, one would need to control the domain of convergence
of a series representation for the grafting map, or to analyze its
analytic continuation.  When using quasi-Fuchsian bending, the failure
of the holonomy map to be a topological covering (\cite{hejhal}) is a potential
obstruction to lifting the bending map to $\P(S)$ beyond a
small neighborhood of a given point in $\T(S)$.

However for our purposes the uniformity of this extension as $X$ and
$\lambda$ vary (i.e.~the existence of the universal constant
$\Extension$) is essential since extension on a smaller neighborhood of
a point in $\T(S)$ corresponds to a larger Lipschitz constant in the
Kobayashi metric argument of \secref{subsec:kobayashi}.  We will
establish such uniformity using a geometric property of $K$--quasi-disks
with $K \approx 1$, however we do not know whether the restriction to
small $K$ is strictly necessary here.  It is natural to ask:
\begin{question}
  Does projective grafting extend to all quasi-Fuchsian groups, or
  further to an open subset $\X(S)$ properly containing $\QF(S)$?
\end{question}
This question can also be interpreted in terms of the domain of
integrability of an incomplete holomorphic vector field on $\P(S)$,
see \cite{goldman:complex-symplectic}.

\subsection{Quasidisks}
The constant $\Extension$ in our proof of \thmref{thm:extension}
comes from the following lemma about quasidisks.  Note that this
constant is independent of the topological type of $S$.

\begin{lemma}
\label{lem:universal-K}
  There exists $\Extension > 0$ with the following property: Let $\Omega$
  be a $K$--quasidisk, where $K < \frac{1+\Extension}{1-\Extension}$.  Let $p,q
  \in \partial \Omega$, and denote by $\gamma(t)$ the Poincar\'e
  geodesic with ideal endpoints $p,q$, parameterized by arc length.
  Then the map $T : \C \to \Hat{\C}$ defined by
$$ T(x + iy) = E_y(p,q) \cdot \gamma(x)$$ is locally
$2$--quasiconformal.
\end{lemma}

Recall from \secref{subsec:bending} that $E_t(p,q)$ is the elliptic
M\"obius transformation fixing $p$ and $q$, and rotating
counterclockwise about $p$ by angle $t$.

\begin{proof}[Proof of \lemref{lem:universal-K}.]
We need to determine a value of $\Extension$ such that $T$ is a local
diffeomorphism and that its dilatation is bounded by $1/3$ (that is,
it is locally $2$--quasiconformal).

In fact we need only study the derivative of $T$ along $y=0$,
since $T(z)$ and $T(z + i y)$ differ by composition with an
elliptic M\"obius transformation, leaving the dilatation (and the property of being
locally diffeomorphic) invariant.

Also note that the condition we wish to establish is invariant under
applying M\"obius transformations to $\Omega$, so we can assume
$p=0$, $q=\infty$ and that $i \in \Omega$ lies on the Poincar\'e geodesic
(parameterized so $\gamma(0) = i$), and we must show that $T$ is
a local diffeomorphism at $0$, and its dilatation at $0$ is bounded by
$1/3$.

Consider the Riemann map $f : \hyp \to \Omega$ normalized to fix
$\{0,i,\infty\}$.  Then we have
$$ \gamma(t) = f \left ( i e^t \right ).$$
By explicit calculation we find
$$
D_0 T = \frac{1}{2}\begin{pmatrix}
-\im f'(i) & -1\\
\re f'(i) & 0\\
\end{pmatrix}
$$
which has dilatation $\mu(0) = \frac{f'(i) - 1}{f'(i) + 1}$.  The
proof will therefore be complete if for some $\Extension$, the normalized Riemann
map satisfies $|f'(i) - 1| < 1/2$ (which, using the formula above,
gives $|\mu(0)| < 1/3$).

Suppose on the contrary that no such $\Extension$ exists.  Then there is a
sequence of $K_n$--quasidisks $\Omega_n$ with $K_n \to 1$, normalized
as above, so that the associated Riemann maps satisfy
\begin{equation}
\label{eq:deriv-close-one}
|f_n'(i) - 1| \geq 1/2 \text{ for all } n.
\end{equation}

Since $\Omega_n$ is obtained from $\hyp$ by applying a
$K_n$--quasiconformal homeomorphism fixing $\{0,i,\infty\}$, the
boundary $\partial \Omega_n$ lies in a $\log(K_n)$--neighborhood of
$\Hat{\R}$ in the Poincar\'e metric of $\Hat{\C} \setminus
\{0,i,\infty\}$ (see \cite[\S 3.D]{ahlfors}), as
pictured in \figref{fig:circlenbd}.  Since $\log(K_n) \to 0$, the
pointed domains $(\Omega_n,i)$ converge to $(\hyp,i)$ in the
Carath\'eodory topology, and thus Riemann maps $f_n$ and their
derivatives converge to the identity uniformly on compact sets.  In
particular $f'_n(i) \to 1$, contradicting \eqref{eq:deriv-close-one}.
This contradiction establishes the lemma.
\end{proof}

\begin{figure}
\begin{center}
\fbox{\includegraphics[width=10cm]{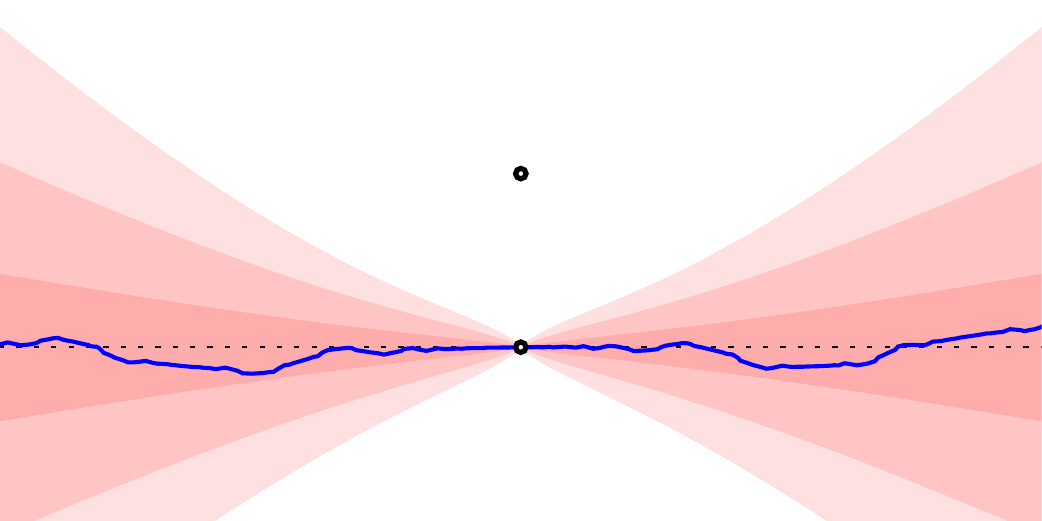}}
\end{center}
\caption{The image of $\R$ under a $K$--quasiconformal deformation of
$\R$
  (normalized to fix $\{0,i,\infty\}$) lies in a
  $\log(K)$--neighborhood of $\R$ with respect to the Poincar\'e metric of $\Hat \C 
\setminus
  \{0,i\}$.
  \label{fig:circlenbd}}
\end{figure}

\subsection{Proof of the extension theorem.}

\begin{proof}[Proof of \thmref{thm:extension}.]
  To fix notation, let $\rho = Q(X,Y) \in \QF_{\Extension}(S)$.
  We construct the extension of grafting and verify its properties in several steps:

\smallskip
\noindent \emph{Step 1: Construction for simple closed curves.}\\
We consider the lamination $t \gamma \in \ML(S)$, where $\gamma$ is a
simple closed geodesic and $t \in \R^+$.  Recall that $\beta_{t
  \gamma,X} : (\hyp \setminus \Tilde{\gamma}) \times (\hyp \setminus
\Tilde{\gamma}) \to \PSL_2(\C)$ is the bending cocycle map for
$(t \gamma, X)$, and as before fix a basepoint $\base \in (\hyp \setminus
\Tilde{\gamma})$.  Abusing notation, we abbreviate $\beta(y) =
\beta_{t \gamma,X}(\base,y)$.

  While $\beta : (\hyp \setminus \Tilde{\gamma}) \to \PSL_2(\C)$ does
  not extend continuously to $\hyp$, there is a natural way to extend it
  to a continuous map $\hat{\beta} : \Tilde{Z} \to \PSL_2(\C)$, where
  $Z = \gr(t \gamma,X)$.  Recall that $\Tilde{Z}$ is obtained from
  $\Tilde{X}$ by replacing each lift of $\gamma$ with a Euclidean
  strip of width $t$ foliated by parallel geodesics.  If $z \in
  \Tilde{Z}$ corresponds to a point $x \in (\Tilde{X} \setminus
  \Tilde{\gamma})$, then we let $\hat{\beta}(z) = \beta(x)$; otherwise
  $z$ belongs to a strip that replaces a lift $g$, and we define
$$ \hat{\beta}(z) = E_s(p,q) \cdot \hat{\beta}(x_0)$$
where $x_0$ is a point in the connected component $P$ of $\Tilde{X}
\setminus \Tilde{\gamma}$ adjacent to $g$ and closer to $\base$,
and where $s$ is the Euclidean distance from $z$ to the edge of the
strip meeting $P$.  Thus, while $\beta(x)$ jumps discontinuously by an
elliptic when $x$ crosses a geodesic lift of $\gamma$ in $\Tilde{X}$,
the extension $\hat{\beta}(z)$ gradually accumulates the same elliptic
as $z$ crosses the associated strip in $\Tilde{Z}$.
  
We can now define the developing map of $\Gr(t \gamma,\rho)$ in terms
of $\hat{\beta}$: Let $f_0 : \Tilde{Z} \to \CP^1$ be the composition
of the lift of the map $Z \to X$ that collapses the grafted cylinder
orthogonally onto $\gamma$ with the Riemann map from $\Tilde{X}$ to the
domain of discontinuity $\Omega$ of $\rho$ covering $X$.  Note that
$f_0$ is holomorphic on the part of $\Tilde{Z}$ coming from $\Tilde{X}
- \Tilde{\gamma}$.  Define
$$
f(z) = \hat{\beta}(z) f_0(z).
$$ By \lemref{lem:universal-K}, the map $f : \Tilde{Z} \to \CP^1$
is a local homeomorphism with complex dilatation $\mu$ satisfying
$|\mu| \leq \frac{1}{3}$.  Identify $\Tilde{Z}$ with the upper
half-plane $\hyp$ equipped with a Fuchsian action of $\pi_1(S)$.  Extend
$\mu$ to $\bar{\hyp}$ by reflection and let $w^\mu : \CP^1 \to \CP^1$ be
the normalized solution to the Beltrami equation.

The quotient of $\hyp$ by the $w^\mu$--conjugated Fuchsian action of
$\pi_1(S)$ gives a Riemann surface $Z^\mu$ (a deformation of $Z$)
and the holomorphic map $f \circ (w^\mu)^{-1} : \hyp \to \CP^1$ is the
developing map of a projective structure $\Gr(t \gamma, \rho)$ on
$Z^\mu$ with holonomy $B_{t\gamma}$.

Thus we have a map $\Gr_{t \gamma} : \QF_{\Extension}(S) \to \P(S)$
satisfying $\hol \circ \Gr_{t \gamma} = B_{t \gamma}$.  Since the
developing map is a holomorphic map from $\Tilde{Z^\mu}$, the
conformal version of this quasi-Fuchsian grafting operation is
$$ \gr(t \gamma,\rho) = \pi \left ( \Gr(t \gamma,\rho) \right ) =
Z^\mu = \left ( \gr(t \gamma, X) \right )^\mu \;\;\text{ where } \; \;
\rho = Q(X,Y).$$  Note that as in the case of Fuchsian grafting along
a simple closed curve, this
grafting operation induces a decomposition of the surface $\gr(t
  \gamma,\rho)$ into a cylinder $A$ and a complementary surface $X_0$
equipped with a conformal isomorphism to $X \setminus \gamma$.
However in the quasi-Fuchsian case, the natural identification of $A$
with $\gamma \times [0,t]$ is only quasiconformal (rather than
conformal).

It is also easy to see that this procedure is a generalization of the
usual projective grafting operation, since when $\rho = Q(X,X)$ is
Fuchsian, the Poincar\'e geodesic joining $0$ to $\infty$ in the
domain of discontinuity is the imaginary axis, the map $T(z) = e^z$ is 
holomorphic, $\mu 
\equiv 0$, $Z^\mu = Z = \gr(t\gamma,X)$, and
$\Gr(t \gamma,Q(X,X)) = \Gr(t \gamma,X)$.

\smallskip
\noindent \emph{Step 2: Continuity and holomorphicity.}\\
  We now analyze the continuity of this extension of grafting as $\rho
  = Q(X,Y)$ is varied in $\QF_\Extension(S)$ (while the lamination $t
  \gamma$ remains fixed).  Recall that Fuchsian grafting is a
  real-analytic map, so $Z = \gr(t \gamma,X)$ varies smoothly with
  $\rho$, as does the grafting cylinder $A \subset Z$.  The Poincar\'e
  geodesic in the domain of discontinuity of $\rho$ also depends
  real-analytically on $\rho$, since the limit set of $\rho$ undergoes
  a holomorphic motion as $\rho$ is varies in $\QF_\Extension(S)$ and the
  Poincar\'e geodesic is the image of a fixed line in $\hyp$ under the
  associated holomorphic family of Riemann maps.  Thus the Beltrami
  coefficient $\mu$, which is supported on the grafting cylinder $A
  \subset Z$, varies smoothly in the interior of $A$.  Combining this
  with the smooth variation of the boundary of $A$ and the continuous
  dependence of solutions of the Beltrami equation on $\mu$, we
  conclude that both the deformed domain surface $Z^\mu$ and the local
  charts of the projective structure vary continuously with $\rho$.
  Thus $\Gr_{t \gamma} : \QF_\Extension(S) \to \P(S)$ and $\gr_{t \gamma}
  : \QF_\Extension(S) \to \T(S)$ are continuous maps.

Since $\hol \circ \Gr_{t \gamma} = B_{t \gamma}$, and $\hol$ is a local
homeomorphism, we can locally express the extension of grafting as
$$\Gr_{t \gamma} = \hol^{-1} \circ B_{t \gamma},$$ where $\hol^{-1}$
is a suitable local branch of the inverse of $\hol$.  Note that the
continuity of $\Gr_{t \gamma}$ ensures that this description is valid
on an open neighborhood of any point in $\QF_\Extension(S)$.  Since $B_{t
  \gamma}$ and $\hol$ are holomorphic maps, it follows that $\Gr_{t
  \gamma} : \QF_\Extension(S) \to \P(S)$ is itself holomorphic.

\smallskip
\noindent \emph{Step 3: Extension to general measured laminations.}\\
For any $\lambda \in \ML(S)$, let $c_n \gamma_n \to
\lambda$ where $c_n \in \R^+$ and $\gamma_n$ are simple closed curves.
To study the convergence of grafting maps, realize $\T(S)$ as a
bounded open set $\Omega \subset \C^N$, which induces an
identification of $\P(S) \simeq T^{1,0}\T(S)$ with the set $\Omega
\times \C^N \subset \C^{2N}$.  In the rest of the proof, we use these
identifications to regard the grafting maps as tuples of holomorphic
functions.

Since the usual grafting operation extends continuously to measured
laminations, the holomorphic maps $\Gr_{c_n \gamma_n} : \QF_\Extension(S)
\to \P(S)$ (and thus also $\gr_{c_n \gamma_n} = \pi \circ \Gr_{c_n
  \gamma_n}$) converge locally uniformly on the set $\F(S)$ of
Fuchsian representations.  To show that $\Gr_{c_n \gamma_n}$ converges
locally uniformly to a holomorphic map $\Gr_\lambda : \QF_\Extension(S)
\to \P(S)$, we need only show that this family of maps is normal,
since any two limit maps of subsequences would then agree on $\F(S)$,
a maximal totally real submanifold, and hence they would agree
throughout $\QF_\Extension(S)$.

Normality is immediate for $\gr_{c_n \gamma_n}$ by the boundedness of
the embedding of $\T(S)$ in $\C^N$, thus these conformal grafting maps
converge locally uniformly to $\gr_{\lambda} : \QF_\Extension(S) \to
\T(S)$.  

\begin{remark}
  We have now established the extension of conformal grafting to
  $\QF_\Extension(S)$.  As mentioned in the introduction, this is the only
  part of \thmref{thm:extension} which is used in the proof of
  \thmref{thm:Lipschitz}.
\end{remark}

Suppose $E \subset \QF_\Extension(S)$ is compact.  We will show that the
restrictions of $\Gr_{c_n \gamma_n}$ to $E$ are uniformly bounded.  Let
  $K \subset \T(S)$ be a compact set in $\T(S)$ containing
$$ F = \bigcup_{n=1}^\infty \gr_{c_n \gamma_n}(E),$$
which exists by uniform convergence of $\gr_{c_n \gamma_n}$ on $E$.
By \lemref{lem:compactness}, in order to construct a compact set
$\Hat{K} \subset \P(S)$ that contains
$$ \Hat{F} = \bigcup_{n=1}^\infty \Gr_{c_n \gamma_n}(E),$$
it suffices to show that all projective structures in $\Hat{F}$ have a
$\delta$--injective disk, for some $\delta > 0$.

Recall that the Riemann surface $\gr(c_n \gamma_n,Q(X,Y))$ is a
$2$--quasiconformal deformation of $\gr(c_n \gamma_n,X)$, and that the
$2$--quasiconformal map $f : \gr(c_n \gamma_n,X) \to \gr(c_n
  \gamma_n,Q(X,Y))$ respects the inclusion of $X \setminus \gamma_n$
into each of these surfaces.  The developing map of $\Gr(c_n
  \gamma_n,Q(X,Y))$ is injective on each connected component of the
lift of $(X \setminus \gamma_n)$ to $\Tilde{\Gr(c_n \gamma_n,
  Q(X,Y))}$.  Since the set of grafted surfaces $\{ X \: | Q(X,Y) \in E
\}$ is compact, and the sequence $c_n \gamma_n$ is convergent in
$\ML(S)$, \corref{cor:injectiveball} provides a uniform radius $r_1$
such that the image of $X \setminus \gamma_n$ in $\gr(c_n \gamma_n,
X)$ contains an $r_1$--ball with respect to the hyperbolic metric on
$\gr(c_n \gamma_n,X)$.

The $2$--quasiconformal map $f: \gr(c_n \gamma_n,X) \to \gr(c_n
\gamma_n,Q(X,Y))$ is uniformly $\frac{1}{2}$--H\"older with respect to
the hyperbolic metric \cite[\S~3C]{ahlfors}, so the image of $X \setminus
\gamma_n$ in $\gr(c_n \gamma_n,Q(X,Y))$ contains a hyperbolic ball of
radius $r_2 = C \sqrt{r_1}$, for a universal constant $C$.  In
particular, the developing map of $\Gr(c_n \gamma_n,Q(X,Y))$ has an
$r_2$--injective disk for all $n \in \N$ and $Q(X,Y) \in
\QF_\Extension(S)$, and we conclude that the sequence of maps
$\Gr_{c_n \gamma_n}$ converges locally uniformly to a holomorphic map
$\Gr_\lambda : \QF_\Extension(S) \to \P(S)$.
\end{proof}

\begin{remark}
  Tanigawa showed that for projective structures on compact surfaces,
  the holonomy map is proper when restricted to $\pi^{-1}(K)$, where
  $K \subset \T(S)$ is any compact set \cite{tanigawa}.  Combined with
  the continuity of the shear-bend map, this provides a shorter (if
  less elementary) alternative to the last three paragraphs of the
  proof of \thmref{thm:extension}.  However, Tanigawa's proof does not
  immediately extend to punctured surfaces, nor do those of the
  similar properness results of Gallo-Kapovich-Marden \cite{GKM}.
\end{remark}

\subsection{The Kobayashi estimate}
\label{subsec:kobayashi}

Using the extension theorem (\ref{thm:extension}), we now prove
the Lipschitz property for the grafting map.

\begin{proof}[Proof of \thmref{thm:Lipschitz}.]
We want to find an upper bound for $d(\gr(\lambda,X), \gr(\lambda
,Y))$. Let $R = \frac{1}{2} \log \frac{1+\Extension}{1-\Extension}$, where
$\Extension$ is as in \thmref{thm:extension}.  It is enough to
prove the Lipschitz property for $X,Y$ such that $d(X,Y) < R/2$, so we
assume this for the rest of the proof.

Let $r_0,r_1 < 1$ be the radii of Euclidean disks concentric with
$\Delta$ that represent hyperbolic disks of radius $R/2$ and $R$,
respectively, in the unit disk model of $\hyp^2$.  Let $\kappa : \Delta
\to \T(S)$ be the \Teich disk such that
\begin{enumerate}
\item $X = \kappa(0)$
\item $Y = \kappa(r)$ for some $r \in \R^+$, $r < r_0$
\end{enumerate}
Note that $Q(X,\kappa(z)) \in \QF_\delta(S)$ for all  $|z| < r_1$

Let $d_{\QF}$ denote the Kobayashi metric on $\QF_{\Extension}(S)$.  Since
$z \mapsto Q(X,\kappa(z))$ is a holomorphic map of $\Delta_{r_1}$
into $\QF_\Extension(S)$, we have $d_{\QF}(Q(X,X),Q(X,Y)) < C d(X,Y)$
where $C$ depends only on $\Extension$ (through $r_0,r_1$).  Since
$\gr_\lambda : \QF_\Extension(S) \to \T(S)$ is holomorphic, it does not
expand the Kobayashi distance, and $$d(\gr(\lambda,X), \gr(\lambda,
Q(X,Y))) < C d(X,Y).$$ Exchanging the roles of $X$ and $Y$ and using the
triangle inequality, we have $d(\gr(\lambda,X), \gr(\lambda,Y)) < 2 C
d(X,Y)$, establishing the theorem for $\Lipschitz = 2 C$.
\end{proof}

\begin{remark}
Since it depends only on $\Extension$ from \lemref{lem:universal-K},
the Lipschitz constant $\Lipschitz$ in \thmref{thm:Lipschitz} is
also universal.  In particular it does not depend on the genus of the
surface $S$.
\end{remark}

\section{Conclusion of Proof}
 \label{sec:Conclusion}
 
In terms of comparing Teichm\"uller geodesic rays and grafting rays,
Theorem \ref{thm:Lipschitz} allows us to freely move the starting
point of a grafting ray by a bounded distance.  We will also need a
similar result of Rafi for Teichm\"uller geodesic rays:

\begin{theorem}[{\cite[\S7]{rafi3}}]
\label{thm:rafi}
For any $\epsilon > 0$ and $\distance > 0$ there exists a constant 
$\Distance>0$ so that for any $X,Y \in \T(S)$ where $X$ is $\epsilon$--thick 
and $d_\calT(X,Y) \leq \distance$ and any
$\lambda \in \ML(S)$, we have
$$
d_\T\big(\G(t,\lambda,X), \G(t, \lambda, Y)\big) < \Distance.
$$
for all $t \geq 0$.
\end{theorem}

We are now ready to prove \thmref{thm:Main} in the case $S$ has
no punctures by combining Proposition \ref{prop:TheMap},
Theorem \ref{thm:Lipschitz}, and Theorem \ref{thm:rafi}.  The action
of the mapping class group will be used to bridge the gap between the
considerations of section \ref{sec:TheMap}, which give uniform
estimates only for certain points $\Xstd \in \T(S)$ and for $\lambda$
in an open set $U \subset \ML(S)$, and the general case of arbitrary
$\lambda$ and any $\epsilon$--thick $X$.

\begin{theorem}
\label{thm:Closed}
Let $S$ be a compact surface, $X \in \T(S)$ and let $\lambda$
be a measured geodesic lamination on $X$ with unit hyperbolic length.
Then, for all $t \geq 0$ we have
$$ 
d_\calT \Big ( \gr \big(e^{2t}\lambda, X \big), \G \big(t,\lambda,X \big) \Big) 
   \leq \FinalBound,
$$
where $\FinalBound$ is a constant depending on $X$ but not on $\lambda$.
\end{theorem}

\begin{proof}
Recall from \secref{sec:Standard} that $U$ and $\bU$ are disjoint open sets in 
$\ML(S)$ containing $\nu$ and $\bnu$, where $[\nu]$ and $[\, \bnu\,]$ are the 
stable laminations for $\varphi$ and $\bvp$ respectively.  We can assume 
the representatives $\nu$ and $\bnu$ are chosen so that they
have unit length on $\Xstd$. Let $[ U]$ and $[\, \bU \,]$ be the images of $U$ 
and $\bU$, under the projection of $\ML(S)$ to $\PML(S)$.  
Then there is a power $n$ such that for every measured 
lamination $\lambda \in \ML(S)$, either 
$$
[\varphi^n(\lambda)] \in [U]
\qquad\text{or}\qquad
[\, \bvp^{\,n}(\lambda) \,] \in [\,\bU\,].
$$
The key point is that the value of $n$ is independently of $\lambda$
and depends only on sets $U$ and $\bU$. For the rest of the proof, assume 
$[\varphi^n(\lambda)] \in [U]$; the case where
$[\, \bvp^{\, n}(\lambda)\,] \in [\, \bU \,]$ can be dealt with similarly. 

Let $c_\lambda$ be a constant such that 
$c_\lambda \cdot \varphi^n(\lambda) \in U$.
Then it follows from \propref{prop:TheMap} that there is a Riemann surface 
$Z_{\varphi^n(\lambda)}$ and a constant $\quasi = \quasi(t_0)$ such that
for all $t \geq t_0$,
$$
d_\T \Big(\gr\big( e^{2t} c_\lambda \varphi^n(\lambda), \Xstd \big),
 \G \big( t,\varphi^n(\lambda), Z_{\varphi^n(\lambda)} \big) \Big) \leq \quasi.
$$
To get rid of $c_\lambda$, choose $t_\lambda$ so that
$e^{2t_\lambda} = c_\lambda$ and reparameterize using the parameter
$(t + t_\lambda)$.  Since $\nu$ has unit length on $\Xstd$, all the measured 
laminations in $U$ have length close to $1$.  Hence 
$$
\ell_{\Xstd}\big(c_\lambda \varphi^n(\lambda)\big) = 
   c_\lambda \ell_{\Xstd}\big(\varphi^n(\lambda)\big)
$$
is close to $1$. On the other hand, $\ell_{\Xstd}(\varphi^n(\param))$,
as a function on  all measured laminations of unit length on $\Xstd$,
attains a maximum and minimum value.
Therefore, the $c_\lambda$ are bounded above and below, independently of 
$\lambda$, and so the same is true of $t_\lambda$.
 Thus, there is a surface $Y_{\varphi^n(\lambda)}$
(the marked conformal structure determined by the pair of
measured foliations
$e^{-t_\lambda} \varphi^n(\lambda)$ and  $e^{t_\lambda} \horizontal (\varphi^n(\lambda), \Xstd$))
and a constant $\quasi$ such that for all $t \geq 0$,
$$ 
d_\T \Big(\gr\big( e^{2t} \varphi^n(\lambda), \Xstd \big),
 \G \big( t,\varphi^n(\lambda), Y_{\varphi^n(\lambda)}\big) \Big) \leq \quasi.
 $$

Let $Y = \varphi^{-n}(  Y_{\varphi^n(\lambda)})$ and $\hat X= \varphi^{-n}(\Xstd)$.
After moving the above \Teich ray and grafting ray by $\varphi^{-n}$ we have
$$
d_\T \Big( \gr \big(e^{2t} \, \lambda, \Xnew \, \big),
  \G \big( t,\lambda,Y\big) \Big) \leq \quasi.
$$
From \thmref{thm:Lipschitz} we have
$$
d_\T \Big(\gr \big( e^{2t} \lambda, X \big), \gr \big(e^{2t} \lambda,
\Xnew\,\big) \Big) 
  \leq \Lipschitz  \ d_\T(X, \Xnew),
$$
and by \thmref{thm:rafi} we have
$$
d_\T \big(\G(t,\lambda,Y), \G(t,\lambda,X)) < \Distance
$$
Now, combining these three inequalities and the triangle inequality, we have:
\begin{equation*}
d_\calT \Big(\gr\big( e^{2t} \lambda, X \big), \G \big( t, \lambda, X) \Big ) \leq 
\Lipschitz \, d_\T 
(X, \Xnew) + \quasi + \Distance.
\end{equation*}
Note that $\Xnew$ is chosen independently of $\lambda$; the same $n$ works for
all $\lambda$.  Therefore, $d_\T(X, \Xnew)$ depends only on
$X$ and choosing
$$
\FinalBound = \Lipschitz \, d_\T (X, \Xnew) + \quasi + \Distance
$$ 
concludes the proof. The constant $\FinalBound$ depends on $X$ only.
\end{proof}
\section{The case of punctures}
\label{sec:Punctures}
Here we sketch how the argument of
\secsref{sec:Dual}{sec:TheMap} can be modified to prove
\propref{prop:TheMap} when the surface $S$ has finitely many punctures.
By truncating small neighborhoods of the punctures and doubling the
resulting surface along its boundary, we obtain a closed surface
$S^D$.  For $X \in \ts$, the basic strategy is to truncate
horoball neighborhoods of the punctures and deform the complement
slightly to a hyperbolic surface with geodesic boundary, and then
geometrically double this across the boundary to get a surface in
$\T(S^D)$.  For a measured lamination $\lambda$ with compact support
on $S$, let $\lambda_D$ be the measured lamination on $S^D$ which is
the union of $\lambda$ and its mirror image.  Then \thmref{thm:Closed}
provides a map between the graftings along $\lambda_D$ of  the double 
and a \Teich geodesic ray in $\T(S^D)$.  One can assure that this map
is symmetric and obtain a map from $\gr(e^{2t} \lambda, X)$ to a
\Teich geodesic.  However, one needs to be careful so that deforming the
surface commutes with grafting.

\subsection{Projection of a geodesic} 
\label{sec:Projection}
We first recall a theorem which we will use here and in \secref{sec:Example}.  

Let $Q$ be a surface of finite genus, possibly with finitely many punctures.
 Let $\Gamma$ be a collection of disjoint, homotopically distinct simple closed curves 
on $Q$.  Let $R$ be the closure of a component of $S \setminus \Gamma$. 
Extend $\Gamma$ to a pants decomposition
and define associated Fenchel-Nielsen length and twist coordinates.
By forgetting the Fenchel-Nielsen length and twist coordinates associated
to the curves in $\Gamma$ but retaining all remaining Fenchel-Nielsen coordinates,
we obtain a projection 
$$\pi_R: \T(Q) \to \T(R).$$
Here, $\T(R)$ is the space of analytically finite, marked conformal structures on the 
interior of $R$ (so the boundary of $R$ is pinched).

Let $q_0$ be a unit area quadratic differential (see \cite{strebel} 
for definition and background information) and let 
$q_t = \begin{bmatrix} e^t & 0 \\ 0 & e^{-t}\\\end{bmatrix} q_0$ be the image
of $q_0$ under the \Teich geodesic flow. Then, the map
$\G : [a,b] \to \T(Q)$ sending $t$ to the underlying conformal structure 
of $q_t$ is a \Teich geodesic in $\T(Q)$. A description of a the
behavior of a \Teich geodesic is given in \cite{rafi3}.  
We recall from \cite{rafi3} the following theorem that gives a sufficient condition
for $\pi_R(\G(t))$ to fellow travel a \Teich geodesic in $\T(R)$. 

Let $\gamma$ be a boundary component of $R$, and let $\beta$ be an 
essential arc in 
$R$ with both endpoints in $\gamma$.  By the $q_t$--length of $\beta$, 
we mean the $q_t$--length of the shortest arc representing $\beta$ that 
starts and ends on a $q_t$--geodesic representative of $\gamma$. 
Denote this length by $\ell_{q_t}(\beta)$. Define
$$
M_t(\gamma, R) = \min_{\beta} \frac{\ell_{q_t}(\beta) }{\ell_{q_t}(\gamma)},
$$
where $\beta$ ranges over all arcs in $R$ with both endpoints on $\gamma$.
If this quantity is large, then there is an annulus round $\gamma$ in $\G(t)$
that has large modulus.  

\begin{theorem}
[Rafi \cite{rafi3}]
\label{thm:expanding}  
Let $\G: [a,b] \to \T(Q)$ be a \Teich geodesic, and let $q_t$ and $R$ be as above.  
 Then there exists a constant $M>0$ such that if 
\begin{equation} \label{eq:M}
M_t(\gamma, R )> M
\end{equation}
for all $t \in [a,b]$ and for every boundary curve $\gamma$ of $R$,
then there is a geodesic $\G_R \from [a,b] \to \T(R)$ such that
$$ d_{\T(R)} \big( \pi_R(\G(t)), \G_R(t) \big)=O(1).$$
Furthermore, if $R$ is a thick component of the
thick-thin decomposition of $\G(t)$ for every $t \in [a,b]$, then the
condition \eqref{eq:M} can be replaced with
\begin{equation} \label{eq:Diam}
\frac{\diam_{q_t}(R)}{\ell_{q_t}(\gamma)}> M.
\end{equation}
\end{theorem}

\subsection{The doubling argument}

Let $\ep > 0$ be a constant smaller than the Margulis constant and
 let $X^T$ be the
surface obtained from $X$ by truncating horoball neighborhoods of
the punctures, which are bounded by horocycles of lengths $\ep$.

\begin{lemma} \label{lem:Truncate}
There is a constant $K_0$ such that the following holds:
For any $X \in \ts$, any measured lamination $\lambda$ of compact support on $S$, and
any sufficiently small $\epsilon$,
there is a hyperbolic surface $X^B$ 
with geodesic boundaries of length $\ep$ and a 
$K_0$--quasiconformal homeomorphism $\phi \from X^T \to X^B$
that sends the geodesic representative of $\lambda$ in $X^T$
isometrically to the geodesic representative of $\lambda$ in $X^B$.
\end{lemma}

\begin{proof}
By adding finitely many leaves, extend the support of $\lambda$ to a
maximal compact geodesic lamination $\Lambda \subset X$.
Note that $\Lambda$ no longer supports a measure. Each
connected component of $X \setminus \Lambda$ that contains a puncture
is isometric to a punctured monogon (the result of
symmetrically gluing two edges of an ideal triangle).
We describe how to construct the desired map on a truncated monogon. 
\begin{figure}
\begin{center}
\includegraphics{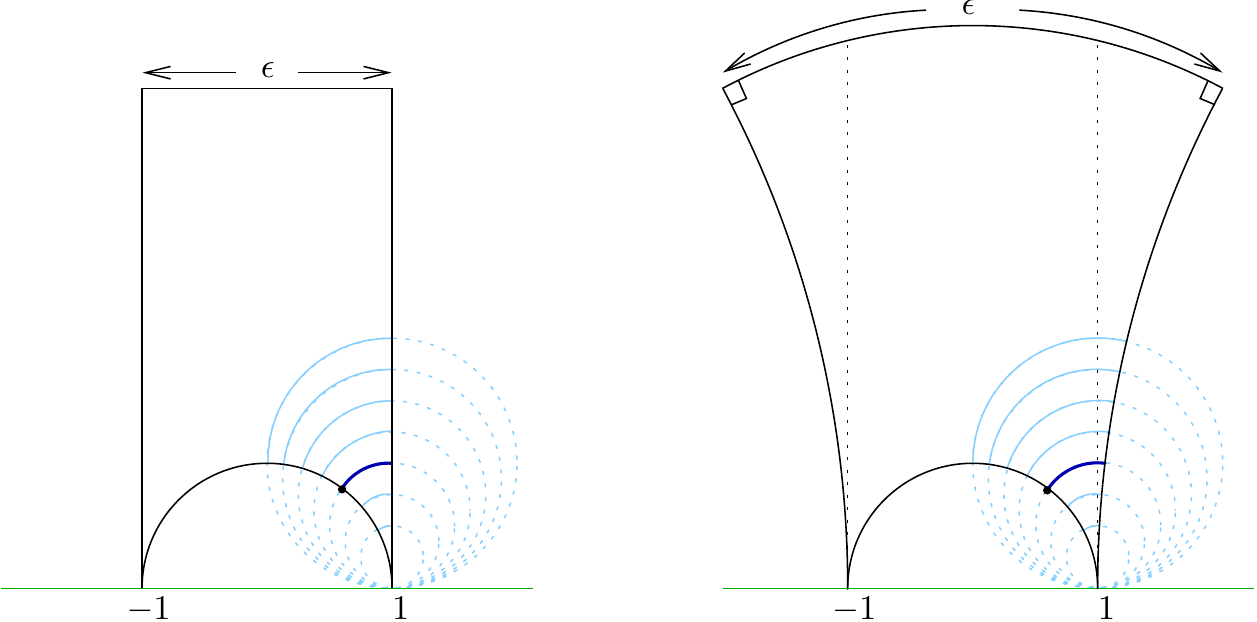}
\end{center}
\caption{On the left is the truncated ideal triangle $P^T$ and on the right is the
polygon $P^B$ to which it is mapped.   A horocyclic segment around $1$ or
$-1$ of length $c < 1$ is mapped to a horocylic segment of length
$c \cosh(\ep/2)$ on the same horocycle.}
\label{fig:monogon}
\end{figure}

Take a punctured monogon and 
cut it into an ideal triangle.  Let $P^T$ be the truncated ideal triangle,
bounded by the horocyclic segment of length $\ep$, as shown in \figref{fig:monogon}.
In the figure, the leaf of $\Lambda$ is represented by
the geodesic $g$ joining $-1$ and $1$.  
Now consider the two geodesics with endpoints at $-1$ and $1$ respectively, that
are symmetric across the geodesic joining $0$ and $\infty$, such that
their common perpendicular has length $\ep$, as shown
on the right in \figref{fig:monogon}.  Let $P^B$ be the polygon bounded
by $g$, these two geodesics, and their common perpendicular. 
To prove the lemma, it 
 is sufficient to show that there is a bi-Lipschitz homeomorphism $h: P^T \to P^B$ which is the identity on $g$.  We need to take care
defining $h$ in the horoball neighborhoods of $-1$ and $1$,
since in the surface $X$, the leaf of $\Lambda$ corresponding to $g$ accumulates. 

  Denote the left and right vertical
edges of $P^T$ by  $g_{-}^T$ and $g_+^T$, respectively,
and denote the left and right edges of $P^B$ 
by  $g_{-}^B$ and $g_+^B$, respectively.
Consider the horocyclic segments of length $1$, that joins $g$ to 
$g_{-}^T$ and $g$ to $g_+^T$. 
Foliate the horoball neighborhoods they bound, with horocyclic segments, as indicated
partially in the figure.  If $J$ is such a segment which joins a point $p$ on $g$ to $g_+^T$,
then define $h$ to map $J$ linearly
onto the horocyclic segment in $P^B$ that joins $p$ to $g_+^B$.
As a result, the portion of $g_+^T$ contained in the horoball neighborhood
$H$ of $1$ is mapped isometrically  (by a parabolic isometry fixing $1$) onto
 $g_+^B \cap H$.  
The analogous statement holds for $g_{-}^T$.  Note that the construction
is symmetric with respect to the geodesic joining $0$ and $\infty$. 

  If $J$ has length $c$, then it follows from elementary calculations that $h(J)$ has
length $c \cosh(\ep/2).$
Thus, for any sufficiently small $\epsilon$, it follows that the map $h$ is bi-Lipschitz on the
two 
horoball neighborhoods so that the Lipschitz constant is uniformly bounded. 

It is easy to extend $h$ to a symmetric bi-Lipschitz map on the remainder of $P^T$; 
a horoball neighborhood of the $\ep$--horocycle in $P^T$ can be mapped 
to a quadrilateral in $P^B$ whose one edge is the common perpendicular in $P^B$ and
whose adjacent edges are contained in $g_-^B$ and $g_+^B$.  The map
can be further extended to the remaining compact part easily.  
\end{proof}

For $t>0$, let
$$
X_t= \gr\big( e^{2t} \lambda, X \big).
$$
The surface $X_t$ can also be truncated to a surface $X_t^T$ and the
map $\phi$ in \lemref{lem:Truncate} extends by identity
to a $K_0$--quasiconformal map $\phi_t$ from $X_t^T$ to a surface
$X_t^B$ with geodesic boundary.
Now double the surface $X^B$ along its boundaries to obtain a closed
surface $X^D$.
Then the marking map $S \to X$ extends naturally (up to a Dehn twists 
around the boundary)
to a homeomorphism $S^D \to X^D$. We fix this marking
map and consider $X^D$ as an element of $\calT(S^D)$.

We argue as in \secref{sec:TheMap}, but this time we choose $\omega$ to 
be a disjoint union of two arcs with a component in each half of $X^D$ that 
is preserved under the reflection.
Then by \thmref{thm:Closed}, for all $t \geq 0$, we have uniformly quasiconformal
maps between
$$
X_t^D= \gr(e^{2t} \lambda_D, X^D)
\qquad\text{and}\qquad
Y_t^D= \G(t, \lambda_D, Y^D)
$$
for some surface $Y^D \in \calT(S^D)$.
 But since all the initial data is
symmetric, from the construction, we can conclude that this map is symmetric
as well. That is, if $Y_t^B$ is one half of $Y_t^D$, then there are
uniformly quasiconformal maps
$$
f_t \from X_t^B \to Y_t^B.
$$

Let $\Gamma$ be the set of curves in $S^D$ preserved by the
reflection, i.e., the curves corresponding to $\partial X^B$.
Every curve in $\Gamma$ has length $\ep$ in the Thurston metric
of $X_t^D$.  Since the hyperbolic metric on $X_t^D$ is pointwise smaller than
the Thurston metric, it follows that every curve in $\Gamma$ has length
less than $\ep$ in $X_t^D$.  Since the distance between $X_t^D$ and $Y_t^D$ is bounded
by some constant $\quasi$,  the curves in $\Gamma$ have length
less than $e^{2 \quasi} \ep$ in  $Y_t^D$ \cite{wolpert}.  In particular, by choosing $\ep$ small enough, we can ensure that the lengths 
of curves in $\Gamma$ are small as we like in  $Y_t^D$.

As discussed above, we have a projection
$$ \pi: \T(S^D) \to \ts$$
which pinches all the curves in $\Gamma$. 
Let $Y_t = \pi(Y_t^D)$.  We can again truncate $Y_t$ to a surface
$Y_t^T$.  It follows from the proof of \lemref{lem:Truncate} that
there is a $K_0$--quasiconformal map $\psi_t \from Y^T_t \to Y^B_t$.
To summarize, we have:
$$
\gr (e^{2t} \lambda, X) = X_t
\stackrel{\text{\tiny cut}}{\leadsto}
X_t^T
\stackrel{\! \phi_t}{\longrightarrow}
X_t^B
\stackrel{\! f_t}{\longrightarrow}
Y_t^B
\stackrel{\:\: \psi_t}{\longleftarrow}
Y_t^T
\stackrel{\text{\tiny glue}}{\leadsto}
Y_t
$$
After gluing back the neighborhoods of the punctures, the map 
$$
\psi_t^{-1} \circ f_t \circ \phi_t : X_t^T \to Y_t^T
$$
can be extended to a quasiconformal map between 
$X_t$ and $Y_t$ whose quasiconformal constants are uniformly bounded for all $t\geq 0$.

It remains to be shown that $Y_t$ fellow travels a 
geodesic in $\calT(S)$.  For this, we use \thmref{thm:expanding}.
Consider the family of quadratic differentials $q_t^D$ associated to the 
geodesic $Y_t^D$ and let $\gamma$ be a curve in $\Gamma$. 
We need to show $M_t(\gamma, S)$ is large. 
Here, $S$ is considered as one component of $S^D \setminus \Gamma$.  Let
$S'$ be the other component.  
Note that since $\gamma$ is disjoint from $\lambda_D$ 
(it is completely vertical) it has a unique geodesic representative
in $q_t$.  Since the hyperbolic length of $\gamma$ is small
in $Y_t^D$, there are a pair of annuli with large
modulus on either side of the $q_t$ geodesic representative of $\gamma$
 \cite{minsky-harmonic}. 
More precisely, we have (see \cite[\S 5]{crs1})
$$ \frac{1}{\ell_{Y_t^D}(\gamma)}
\asymp
\log \max \big\{ M_t(\gamma, S), M_t(\gamma, S') \big\},$$
but by symmetry, the right-hand side can be replaced by 
$\log M_t(\gamma, S)$.
Hence, if $\ep$ is sufficiently small, then
$M_t(\gamma, S)$ is sufficiently large.  Thus it follows from 
 \thmref{thm:expanding} that
$Y_t$, $t \geq 0$, fellow
travels a \Teich geodesic. This finishes the proof.
\section{Example showing that theorem is sharp}
\label{sec:Example}
In this section we prove \thmref{introthm:Example}.  
We will use Minsky's product region theorem as stated below.
\subsection{Product region theorem}
Let $S$ be a surface of finite genus, possibly with finitely many punctures.
 Let $\Gamma$ be a collection of disjoint, homotopically distinct simple closed curves 
on $S$ and let $R$ be a component of $S\setminus \Gamma$. As discussed in \secref{sec:Projection},
we have a projection
$\pi_R: \T(S) \to \T(R)$.  In addition,
for each $\gamma \in \Gamma$, take
$\mathbb H_\gamma$ to be a copy of the hyperbolic upper-half plane, and define
$\pi_\gamma: \ts \to \mathbb H_\gamma$ to be
$$\pi_\gamma(X) = s_\gamma(X) + i/{\ell_X(\gamma)},$$
where $s_\gamma$ is the Fenchel-Nielsen twist coordinate associated to $\gamma$. 
Let $\ep > 0$ be a constant smaller than the Margulis constant and
let $\T_{\text{thin}}(\Gamma,\ep) \subset \ts$
 be the subset in which all curves $\gamma \in
\Gamma$ have hyperbolic length at most $\ep > 0$. 
\begin{theorem}[Minsky \cite{minskyprod}]
 \label{thm:prod}
For $\ep$ sufficiently small, if $X, Y \in \T_{\text{thin}}({\mathcal A},\ep)$, 
then
$$
d_{\ts}(X,Y) \eadd
\max_{R, \gamma} \big\{ d_{\T(R)} \big( \pi_R(X),\pi_R(Y),\;
d_{\mathbb H_\gamma} (\pi_\gamma(X), \pi_\gamma(Y)) \big\} 
$$
where the additive constant depends only on $\ep$ and the topological type of $S$. 
\end{theorem}

\subsection{Construction of the example}

\begin{the_example}
There exists a sequence of
  points $X_n$ in $\ts$ and measured laminations $\lambda_n$ with unit
  hyperbolic length on $X_n$ such that for any sequence $Y_n$ in
  $\ts$,
$$\sup_{n,{t \geq 0}} d_\calT \left(\gr(e^{2t} \lambda_n, X_n), \G(t,\lambda_n,Y_n) \right ) 
= \infty.$$
\end{the_example}

\begin{proof}
Let $S$ be a surface of genus $2$. First we construct the sequences $X_n$ 
and $\lambda_n$.  Let $\gamma$ be a separating curve on $S$ and
denote the components of $S \setminus \gamma$ by
$R$ and $R'$. Fix a pair of curves $\alpha$ and $\beta$ in $R$ that intersect
exactly once and a pair of curves $\alpha'$ and $\beta'$ in $R'$
intersecting exactly once. Let $X_n$ be any hyperbolic surface where 
$$\ell_{X_n}(\gamma) = 1/n \quad \text{and} \quad
\ell_{X_n}(\alpha)\emul \ell_{X_n}(\beta) \emul
\ell_{X_n}(\alpha')\emul \ell_{X_n}(\beta') \emul 1.$$
In particular, this implies that $R$ and $R'$ are thick parts in the thick-thin
decomposition of $X_n$ for all sufficiently large $n$. 

Now choose $\lambda$ and $\lambda'$ to be measured laminations with supports 
in $R$ and $R'$ respectively, so that 
$$\ell_{X_n}(\lambda)= \ell_{X_n}(\lambda')=1.$$
Note that this implies in particular that the intersection numbers
$\I (\alpha, \lambda), \I(\beta,\lambda)$ are bounded above. 
We also assume that $\lambda$ and $\lambda'$ are co-bounded.  
That is, the relative twisting (see for example \cite[\S 4.2]{crs1})
of $\lambda$ and $\alpha$,
 and that of $\lambda$ and $\beta$ around any curve in $R$ is uniformly bounded.
 And assume that the analogous statement holds for $\alpha', \beta'$ and
$\lambda'$ in $R'$.
Define
$$
\lambda_n= \left(\frac 1n\right) \lambda+ \left(\frac {n-1}n \right) \lambda'.
$$

Now we examine the grafting ray $\gr(e^{2t}\lambda_n,X_n)$ and will show that 
at $t=(\log n)/2$ the hyperbolic metric of $\gr(e^{2t}\lambda_n,X_n)$, when
restricted to $R$, does not differ much from the metric of $X_n$.
For convenience, let $\gr_n=\gr(n\lambda_n, X_n)$.  Recall that for any curve
$\delta$, its hyperbolic length on $\ell_{gr_n}(\delta)$ on $\gr_n$
is less than its  length in the Thurston metric.  And, the length of $\delta$ in the Thurston
metric is less than
$\ell_{X_n}(\delta) + n\,  \I(\delta, \lambda_n)$  \cite{mcmullen}. Therefore, 
$$
\ell_{gr_n}(\alpha) \leq \ell_{X_n}(\alpha)+ n \I(\alpha, \lambda_n) 
\lmul  1 + n \cdot \frac 1n \I(\alpha, \lambda) \lmul 1. 
$$
Similarly,
$$
\ell_{gr_n}(\beta) \prec 1.
$$
Since $\alpha$ and
$\beta$ intersect once, an upper-bound for the length of one provides
a lower-bound for the length of the other. Hence,
$$
\ell_{\gr_n}(\alpha) \emul \ell_{gr_n}(\beta) \emul 1.
$$
This implies that the restrictions of the hyperbolic
metrics of $X_n$ and $gr_n$ to $R$ are not far apart.  Then it follows from
\thmref{thm:prod} that:
\begin{equation} \label{eq:X-gr}
d_{\calT(R)} \big(\pi(X_n), \pi(gr_n) \big) = O(1),
\end{equation}
where $\pi:  \T(S) \to \T(R)$ is the projection as defined above.

Let $Y_n$ be any sequence of points in \Teich space. 
If $d_\T(X_n, Y_n) \to \infty$, then we are done. Otherwise, there is a constant 
$\Haus$ such that
$$\sup_n d_\T(X_n, Y_n) \leq \Haus.$$ 
We examine the behavior of the \Teich geodesic $\G(t, \lambda_n, Y_n)$
and will use \thmref{thm:expanding} to show that for 
$\G_n = \G((\log n)/2, \lambda_n, Y_n)$ we have
$$
d_{\T(R)} \big(\pi(Y_n), \pi(\G_n) \big) \eadd \frac{\log n}{2}.
$$

First, on $Y_n$,  for any simple closed curve $\delta$  
we have  \cite{wolpert}
$$ 
e^{-2 \Haus} \leq \frac{\ell_{Y_n} (\delta)}{\ell_{X_n}(\delta)} 
\leq e^{2 \Haus},$$
so that 
\begin{equation}
\label{eqn:comparable}
\ell_{X_n}(\delta) \emul \ell_{Y_n}(\delta)
\end{equation}
with multiplicative constants depending on $\Haus$.
In particular this implies that $\ell_{Y_n}(\gamma) \emul 1/n$, and $R$ and $R'$ are
thick components in the thick-thin decomposition of $Y_n$ for all sufficiently large $n$.  Furthermore, because
$\lambda$ and $\lambda'$ were chosen to be co-bounded, it follows from \cite{rafi1} 
that along the geodesic ray $\G(t,\lambda_n, Y_n)$, $t \geq 0$, the curve $\gamma$ is the only
curve that is very short, and $R$ and $R'$ remain thick.  

Let $q_{t,n}$ be the unit-area quadratic differential on  $\G(t, \lambda_n, Y_n)$ 
with vertical foliation in the class of $\lambda_n$. 
Consider the representatives of $R$ and $R'$ with $q_{t,n}$--geodesic boundaries
and let $d_{t,n}$ and $d_{t,n}'$ be their $q_{t,n}$--diameters respectively.  
In order to apply \thmref{thm:expanding}, we need to show that 
$d_{t,n}/\ell_{q_{t,n}}(\gamma)$ is very large.  
For convenience, let $q = q_{0,n}$, $d=d_{0,n}$, and $d' = d'_{0,n}$. 
Since $\gamma$
is disjoint from $\lambda_n$, its $q$--geodesic representative is unique
and is a union of vertical saddle connections.  We have 
(see \cite[\S 5.2]{crs1})
$$
n \asymp \frac{1}{\ell_{Y_n}(\gamma)} \asymp
\log \frac{\max\{d, d' \}}{\ell_{q}(\gamma)}.
$$

If $d \geq d'$, then it follows that
$$ n \asymp \log \frac{d}{\ell_q(\gamma)}.$$
If $d \leq d'$, then 
let $\eta$ denote the horizontal foliation of $q$.  Since $R'$ is thick in $Y_n$, 
it follows from \cite{rafi2} and \eqref{eqn:comparable} that
\begin{equation*}
\ell_{q}(\lambda') \emul d' \,\ell_{Y_n}(\lambda') \emul d'\, \ell_{X_n}(\lambda').
\end{equation*}
Then we have
$$ 1=\I (\lambda_n, \eta) > \frac{n-1}{n}\ell_q(\lambda') \emul d' \ell_{X_n}(\lambda')
\emul d'.$$
Hence we get 
$$n\asymp \log \frac{\max\{d, d' \}}{l_{q}(\gamma)} \prec \log \frac{1}{\ell_{q}(\gamma)}.$$
Therefore, $\ell_{q}(\gamma) \lmul e^{-O(n)}$.
Thus, for $n$ large enough,
$d/\ell_{q}(\gamma)$ is large. 

Since the horizontal length of $\gamma$ is zero, 
$\ell_{q_{t,n}}(\gamma) = e^{-t} \ell_q(\gamma)$.
On the other hand, since $d_{t,n}$ can decrease at most 
exponentially fast (in fact, in this example, it is basically constant), 
we have $d_{t,n} \gmul e^{-t} d$.  Thus $d_{t,n}/\ell_{q_{t,n}}(\gamma)$ remains 
large for all $t \geq 0$, as desired.  

Then, it follows from \thmref{thm:expanding} that for all $t \geq 0$
$$ 
  d_{\T(R)} \big( \pi(Y_n), \pi(\G(t, \lambda_n, Y_n)\big) \eadd t.
$$
In particular, 
\begin{equation} \label{eq:Y-G}
  d_{\T(R)} \big(\pi(Y_n), \pi(\G_n) \big) \eadd \frac{\log n}{2}.
\end{equation}
By \thmref{thm:prod}, $d_\T(X_n, Y_n) \leq \Haus$ implies
\begin{equation}\label{eq:Y-X}
  d_{\T(R)} \big(\pi(Y_n), \pi_R(X_n)\big) = O(1).
\end{equation}
Also by \thmref{thm:prod}, we have
\begin{equation} \label{eq:left}
d_\T(\gr_n, \G_n) \gadd d_{\T(R)}\big(\pi(\gr_n), \pi(\G_n)\big) 
\end{equation}
Now, applying the triangle inequality and using Equations 
\eqref{eq:X-gr}, \eqref{eq:Y-G}, \eqref{eq:Y-X} and \eqref{eq:left} we have
$$
d_\T(\gr_n, \G_n) \gadd \frac{\log n}2,
$$
which goes to infinity as $n \to \infty$. This finishes the proof. 
\end{proof}


\begin{thebibliography}{GKM}

\bibitem[Ahl]{ahlfors}
L.~Ahlfors.
\newblock {\em Lectures on quasiconformal mappings}.
\newblock Manuscript prepared with the assistance of Clifford J. Earle, Jr. Van
  Nostrand Mathematical Studies, No. 10. D. Van Nostrand Co., Inc., Toronto,
  Ont.-New York-London, 1966.

\bibitem[CRS]{crs1}
Y.-E. Choi, K.~Rafi, and C.~Series.
\newblock {Lines of minima and {T}eichm\"uller geodesics}.
\newblock {\em Geom. Funct. Anal.} {\bf 18}(2008), 698--754.

\bibitem[DK]{diaz-kim}
R.~D{\'i}az and I.~Kim.
\newblock {{A}symptotic behavior of grafting rays}.
\newblock Preprint.
\newblock  {{\tt arXiv:0709.0638}}

\bibitem[D]{survey}
D.~Dumas.
\newblock {{C}omplex {P}rojective {S}tructures}.
\newblock In {\em Handbook of {T}eichm{\"u}ller {T}heory, {V}olume {II}}, pages
  455--508. EMS Publishing House, Z\"urich, 2009.

\bibitem[DW]{dumas-wolf}
D.~Dumas and M.~Wolf.
\newblock {Projective structures, grafting, and measured laminations}.
\newblock {\em Geometry and Topology} {\bf 12}(2008), 351--386.

\bibitem[EM]{epstein-marden}
D.~Epstein and A.~Marden.
\newblock {Convex hulls in hyperbolic space, a theorem of {S}ullivan, and
  measured pleated surfaces}.
\newblock In {\em Analytical and geometric aspects of hyperbolic space
  (Coventry/Durham, 1984)}, volume 111 of {\em London Math. Soc. Lecture Note
  Ser.}, pages 113--253. Cambridge Univ. Press, Cambridge, 1987.

\bibitem[FLP]{FLP}
A.~Fathi, F.~Laudenbach, and V.~et~al Poenaru.
\newblock {\em Travaux de {T}hurston sur les surfaces}, volume~66 of {\em
  Ast\'erisque}.
\newblock Soci\'et\'e Math\'ematique de France, Paris, 1979.
\newblock S\'eminaire Orsay, With an English summary.

\bibitem[GKM]{GKM}
D.~Gallo, M.~Kapovich, and A.~Marden.
\newblock {The monodromy groups of {S}chwarzian equations on closed {R}iemann
  surfaces}.
\newblock {\em Ann. of Math. (2)} {\bf 151}(2000), 625--704.

\bibitem[GT]{gilbarg-trudinger}
D.~Gilbarg and N.~Trudinger.
\newblock {\em Elliptic partial differential equations of second order}, volume
  224 of {\em Grundlehren der Mathematischen Wissenschaften}.
\newblock Springer-Verlag, Berlin, second edition, 1983.

\bibitem[Gol]{goldman:complex-symplectic}
W.~Goldman.
\newblock {The complex-symplectic geometry of {${\rm SL}(2,\Bbb C)$}-characters
  over surfaces}.
\newblock In {\em Algebraic groups and arithmetic}, pages 375--407. Tata Inst.
  Fund. Res., Mumbai, 2004.

\bibitem[Gun]{gunning}
R.~Gunning.
\newblock {Affine and projective structures on {R}iemann surfaces}.
\newblock In {\em Riemann surfaces and related topics: Proceedings of the 1978
  Stony Brook Conference (State Univ. New York, Stony Brook, N.Y., 1978)},
  volume~97 of {\em Ann. of Math. Stud.}, pages 225--244, Princeton, N.J.,
  1981. Princeton Univ. Press.

\bibitem[HL]{han-lin}
Q.~Han and F.~Lin.
\newblock {\em Elliptic partial differential equations}, volume~1 of {\em
  Courant Lecture Notes in Mathematics}.
\newblock New York University Courant Institute of Mathematical Sciences, New
  York, 1997.

\bibitem[Hej]{hejhal}
D.~Hejhal.
\newblock {Monodromy groups and linearly polymorphic functions}.
\newblock {\em Acta Math.} {\bf 135}(1975), 1--55.

\bibitem[HP]{heusener-porti}
M.~Heusener and J.~Porti.
\newblock {The variety of characters in {${\rm PSL}\sb 2(\Bbb C)$}}.
\newblock {\em Bol. Soc. Mat. Mexicana (3)} {\bf 10}(2004), 221--237.

\bibitem[HM]{hubbard-masur}
J.~Hubbard and H.~Masur.
\newblock {Quadratic differentials and foliations}.
\newblock {\em Acta Math.} {\bf 142}(1978), 221--274.

\bibitem[Hub]{huber}
A.~Huber.
\newblock {On subharmonic functions and differential geometry in the large}.
\newblock {\em Comment. Math. Helv.} {\bf 32}(1957), 13--72.

\bibitem[Iva]{ivanov}
N.~Ivanov.
\newblock {Isometries of {T}eichm\"uller spaces from the point of view of
  {M}ostow rigidity}.
\newblock In {\em Topology, ergodic theory, real algebraic geometry}, volume
  202 of {\em Amer. Math. Soc. Transl. Ser. 2}, pages 131--149. Amer. Math.
  Soc., Providence, RI, 2001.

\bibitem[KT]{kamishima-tan}
Y.~Kamishima and S.~Tan.
\newblock {Deformation spaces on geometric structures}.
\newblock In {\em Aspects of low-dimensional manifolds}, volume~20 of {\em Adv.
  Stud. Pure Math.}, pages 263--299. Kinokuniya, Tokyo, 1992.

\bibitem[Ker1]{kerckhoffasymp}
S.~Kerckhoff.
\newblock {Asymptotic geometry of {T}eichm\"uller space}.
\newblock {\em Topology} {\bf 19}(1980), 23--41.

\bibitem[Ker2]{kerckhofflom}
S.~Kerckhoff.
\newblock {Lines of minima in {T}eichm\"uller space}.
\newblock {\em Duke Math. J.} {\bf 65}(1992), 187--213.

\bibitem[Kou1]{kourouniotis:bending}
C.~Kourouniotis.
\newblock {Bending in the space of quasi-{F}uchsian structures}.
\newblock {\em Glasgow Math. J.} {\bf 33}(1991), 41--49.

\bibitem[Kou2]{kourouniotis:continuity}
C.~Kourouniotis.
\newblock {On the continuity of bending}.
\newblock In {\em The Epstein birthday schrift}, volume~1 of {\em Geom. Topol.
  Monogr.}, pages 317--334 (electronic). Geom. Topol. Publ., Coventry, 1998.

\bibitem[Mas]{masur}
H.~Masur.
\newblock {Uniquely ergodic quadratic differentials}.
\newblock {\em Comment. Math. Helv.} {\bf 55}(1980), 255--266.

\bibitem[MM]{masur-minsky}
H.~Masur and Y.~Minsky.
\newblock {Geometry of the complex of curves I: Hyperbolicity}.
\newblock {\em Invent. Math.} {\bf 35}(1992), 151--217.

\bibitem[McM]{mcmullen}
C.~McMullen.
\newblock {Complex earthquakes and {T}eichm\"uller theory}.
\newblock {\em J. Amer. Math. Soc.} {\bf 11}(1998), 283--320.

\bibitem[Min1]{minskyprod}
Y.~Minsky.
\newblock {Extremal length estimates and product regions in Teichm\"uller
  space}.
\newblock {\em Duke Math J.} {\bf 83}(1996), 249--286.

\bibitem[Min2]{minsky-harmonic}
Y.~Minsky.
\newblock {Harmonic maps, length, and energy in Teichm\"uller space}.
\newblock {\em J. Differential Geom.} {\bf 138}(1999), 103--149.

\bibitem[MS]{morgan-shalen}
J.~Morgan and P.~Shalen.
\newblock {Valuations, trees, and degenerations of hyperbolic structures. {I}}.
\newblock {\em Ann. of Math. (2)} {\bf 120}(1984), 401--476.

\bibitem[Neh]{nehari}
Z.~Nehari.
\newblock {The {S}chwarzian derivative and schlicht functions}.
\newblock {\em Bull. Amer. Math. Soc.} {\bf 55}(1949), 545--551.

\bibitem[Raf1]{rafi1}
K.~Rafi.
\newblock {A characterization of short curves of a {T}eichm\"uller geodesic}.
\newblock {\em Geom. Topol.} {\bf 9}(2005), 179--202 (electronic).

\bibitem[Raf2]{rafi2}
K.~Rafi.
\newblock {Thick-thin decomposition of quadratic differentials}.
\newblock {\em Math. Res. Lett.} {\bf 14}(2007), 333--341.

\bibitem[Raf3]{rafi3}
K.~Rafi.
\newblock {Hyperbolicity in {T}eichm\"uller space}.
\newblock {\em arXiv:1011.6004} (2010).

\bibitem[SW]{scannell-wolf}
K.~Scannell and M.~Wolf.
\newblock {The grafting map of {T}eichm\"uller space}.
\newblock {\em J. Amer. Math. Soc.} {\bf 15}(2002), 893--927 (electronic).

\bibitem[Str]{strebel}
K.~Strebel.
\newblock {\em Quadratic differentials}, volume~5 of {\em Ergebnisse der
  Mathematik und ihrer Grenzgebiete (3) [Results in Mathematics and Related
  Areas (3)]}.
\newblock Springer-Verlag, Berlin, 1984.

\bibitem[Tan]{tanigawa}
H.~Tanigawa.
\newblock {Grafting, harmonic maps and projective structures on surfaces}.
\newblock {\em J. Differential Geom.} {\bf 47}(1997), 399--419.

\bibitem[Thu]{thurston:minimal-stretch}
W.~Thurston.
\newblock {Minimal stretch maps between hyperbolic surfaces}.
\newblock Unpublished preprint, 1986.
\newblock  {{\tt arXiv:math.GT/9801039}}

\bibitem[Wei]{weiss}
H.~Weiss.
\newblock {The geometry of measured geodesic laminations and measured train
  tracks}.
\newblock {\em Ergodic Theory Dynam. Systems} {\bf 9}(1989), 587--604.

\bibitem[Wol]{wolpert}
S.~Wolpert.
\newblock {The length spectra as moduli for compact Riemann surfaces}.
\newblock {\em Ann. of Math.} {\bf 109}(1979), 323--351.

\end{thebibliography}
\end{document}